\documentclass[
final
leqno
]{siamltex}
\title{Preconditioned Recycling Krylov subspace methods for~self-adjoint problems}
\author{%
Andr\'e Gaul\thanks{Institut f\"ur Mathematik, Technische Universit\"at Berlin, Stra\ss e des 17.\ Juni, D-10623 Berlin, Germany (\texttt{gaul@math.tu-berlin.de}). The work of Andr\'e Gaul was supported by the DFG For\-schungs\-zen\-trum MATHEON.}
\and
Nico Schl\"omer\thanks{Departement Wiskunde en Informatica, Universiteit Antwerpen, Middelheimlaan 1, B-2020 Antwerpen, Belgium (\texttt{nico.schloemer@ua.ac.be}). The work of Nico Schl\"omer was supported by the Research Foundation Flanders (FWO).}}

\usepackage[l2tabu, orthodox]{nag}

\usepackage[tableposition=top]{caption}

\usepackage[font=normalsize]{subcaption}
\usepackage{pgfplots}
\pgfplotsset{compat=newest}

\usepackage{amsmath}
\usepackage{amsfonts}
\usepackage{color}
\usepackage{url}
\usepackage{tabularx}
\usepackage{booktabs}
\usepackage{algorithm} % algorithm floating env
\usepackage{algpseudocode} % algorithmicx

\usepackage[normalem]{ulem} % used for strikeout \sout
\usepackage{bm}
\newcommand\operatorinf[1]{\ensuremath{\mathcal{#1}}}
\newcommand\oiA{\operatorinf{A}}
\newcommand\oiB{\operatorinf{B}}
\newcommand\oiI{\operatorinf{I}}
\newcommand\oiL{\operatorinf{L}}
\newcommand\oiM{\operatorinf{M}}
\newcommand\oiP{\operatorinf{P}}

\newcommand\operatorfin[1]{\ensuremath{\bm{#1}}}

\newcommand\ofB{\operatorfin{B}}

\newcommand\ofD{\operatorfin{D}}
\newcommand\ofE{\operatorfin{E}}
\newcommand\ofF{\operatorfin{F}}
\newcommand\ofG{\operatorfin{G}}
\newcommand\ofI{\operatorfin{I}}

\newcommand\ofT{\operatorfin{T}}

\newcommand\ofZ{\operatorfin{Z}}

\newcommand\rev[1]{#1}

\newcommand\vect[1]{\ensuremath{#1}}
\newcommand\vb{\vect{b}}
\newcommand\ve{\vect{e}}
\newcommand\vr{\vect{r}}
\newcommand\vs{\vect{s}}
\newcommand\vu{\vect{u}}
\newcommand\vv{\vect{v}}
\newcommand\vw{\vect{w}}
\newcommand\vx{\vect{x}}
\newcommand\vy{\vect{y}}
\newcommand\vz{\vect{z}}

\newcommand\vecttuple[1]{\ensuremath{#1}}
\newcommand\vtC{\vecttuple{C}}
\newcommand\vtU{\vecttuple{U}}
\newcommand\vtV{\vecttuple{V}}
\newcommand\vtW{\vecttuple{W}}
\newcommand\vtX{\vecttuple{X}}
\newcommand\vtY{\vecttuple{Y}}

\newcommand\vectspace[1]{\ensuremath{#1}}
\newcommand\vsH{\vectspace{H}}
\newcommand\vsS{\vectspace{S}}
\newcommand\vsL{\vectspace{L}}

\newcommand\A{\ensuremath{\bm{A}}}
\newcommand\0{\ensuremath{\bm{0}}}

\newcommand\x{\ensuremath{\bm{x}}}
\newcommand\ee{\ensuremath{\bm{e}}}
\newcommand\w{\ensuremath{\bm{w}}}
\newcommand\dfn{\ensuremath{\mathrel{\mathop:}=}}
\newcommand\igralnl[4]{\ensuremath{\int\nolimits_{#1}^{#2} #3 \, \mathrm{d} #4}}
\newcommand\ii{\ensuremath{\textup{i}}}
\newcommand\conj[1]{\ensuremath{\overline{#1}}}
\newcommand\bn{\ensuremath{\bm{\nabla}}}
\newcommand\n{\ensuremath{\mathbf{n}}}
\newcommand\R{\ensuremath{\mathbb{R}}}
\newcommand\C{\ensuremath{\mathbb{C}}}
\newcommand\N{\ensuremath{\mathbb{N}}}

\newcommand\B{\ensuremath{\bm{B}}}

\newcommand\J{\ensuremath{\mathcal{J}}}
\newcommand\K{\ensuremath{\mathcal{K}}}
\newcommand\Sc{\ensuremath{\mathcal{S}}}
\newcommand\KK{\ensuremath{\mathbb{K}}}
\newcommand\tp{\ensuremath{\mathrm{T}}}
\newcommand\htp{\ensuremath{\mathrm{H}}}
\newcommand\inv{\ensuremath{{-1}}}
\newcommand\idx{\ensuremath{{(k)}}}
\newcommand\til[1]{\ensuremath{\widetilde{#1}}}
\newcommand{\mat}[1]{\begin{bmatrix}#1\end{bmatrix}}

\newcommand\lra{\ensuremath{\longrightarrow}}
\newcommand\LLRA{\ensuremath{\Longleftrightarrow}}

\newcommand{\ip}[3][]{\ensuremath{{\left\langle{#2},{#3}\right\rangle}_{#1}}}
\newcommand{\ipdots}[1][]{\ensuremath{{\left\langle{\cdot},{\cdot}\right\rangle}_{#1}}}
\newcommand{\nrm}[2][]{\ensuremath{\left\|#2\right\|_{#1}}}
\newcommand{\DEF}{\ensuremath\mathrel{\mathop:}=}

\DeclareMathOperator{\spn}{span}
\DeclareMathOperator{\range}{range}
\newlength\figurewidth
\newlength\figureheight
\usepackage{amsthm}
\theoremstyle{plain}
\newtheorem{thm}{Theorem}
\newtheorem{lem}[thm]{Lemma}
\newtheorem{cor}[thm]{Corollary}
\theoremstyle{definition}
\newtheorem{dftn}{Definition}
\theoremstyle{remark}
\newtheorem{rmk}{Remark}
\newtheorem{setup}{Test setup}
\begin{document}

\maketitle

\begin{abstract}
The authors propose a recycling Krylov subspace method for the solution of a
sequence of self-adjoint linear systems.  Such problems appear, for example, in
the Newton process for solving nonlinear equations.  Ritz vectors are
automatically extracted from one MINRES run and then used for self-adjoint
deflation in the next.  The method is designed to work with arbitrary inner
products and arbitrary self-adjoint positive-definite preconditioners whose
inverse can be computed with high accuracy.  Numerical experiments with
nonlinear Schr\"odinger equations indicate a substantial decrease in
computation time when recycling is used.
\end{abstract}



\pagestyle{myheadings}
\thispagestyle{plain}

\section{Introduction}
Sequences of linear algebraic systems frequently occur in the numerical
solution process of various kinds of problems. Most notable are implicit time
stepping schemes and Newton's method for solving nonlinear equation systems. It
is often the case that the operators in subsequent linear systems  have similar
spectral properties or are in fact equal. To exploit this, a common approach is
to factorize the operator once and apply the factorization to the following
steps. However, this strategy typically has high memory requirements and is
thus hardly applicable to problems with many unknowns. Also, it is not
applicable if subsequent linear operators are even only slightly different from
each other.

The authors use the idea of an alternative approach that carries over spectral
information from one linear system to the next by extracting approximations of
eigenvectors and using them in a deflation
framework~\cite{ChaS97,Mor95,Mor05,SaaYEG00}.  For a more detailed overview on
the background of such methods, see~\cite{GauGLN13}.  The method is designed
for Krylov subspace methods in general and is worked out in this paper for the
MINRES method~\cite{PaiS75} in particular.

The idea of recycling spectral information in Krylov
subspace methods is not new.
Notably, Kilmer and de~Sturler~\cite{KilS06} adapted the GCRO
  method~\cite{Stu96} for recycling in the setting of a sequence of linear
  systems. Essentially, this strategy consists of applying the MINRES method to
  a projected linear system where the projection is built from approximate
  eigenvectors for the first matrix of the sequence. Wang, de Sturler, and
  Paulino proposed the RMINRES method~\cite{WanSP07} that also includes the
  extraction of approximate eigenvectors. In contrast to Kilmer and de Sturler,
  the RMINRES method is a modification of the MINRES method that explicitly
  includes these vectors in the search space for the following linear systems
  (\emph{augmentation}).  Similar recycling techniques based on GCRO have also
  been used by Parks et al.~\cite{ParSMJM06}, Mello et al.~\cite{MelSPS10},
  Feng, Benner, and Korvink~\cite{FenBK13} and Soodhalter, Szyld, and
  Xue~\cite{SooSX13}.  A different approach has been proposed by Giraud,
Gratton, and Martin~\cite{GirGM07}, where a preconditioner is updated with
approximate spectral information for use in a GMRES variant.

GCRO-based methods with augmentation of the search space, including
  RMINRES, are mathematically equivalent to the standard GMRES method (or MINRES
  for the symmetric case) applied to a projected linear system~\cite{GauGLN13}.
Krylov subspace methods that are applied to projected
linear systems are often called \emph{deflated} methods. In the literature,
both augmented and deflated methods have been used in a variety of settings; we
refer to Eiermann, Ernst, and Schneider~\cite{EieES00} and the review
article by Simoncini and Szyld~\cite{SimS07} for a comprehensive overview.

% differences:
% * mathematical idea clearer than in previous approaches (e.g., Wang) by
%   separating the deflation from the Krylov subspace method.
%   -> allows arbitrary preconditioning with deflated MINRES (not only split preconditioners)
%   -> arbitrary inner product
%   -> any already existing MINRES implementation can be used given the projection.
% * we use an oblique projection (like the one used for deflated CG) instead of an orthogonal projection (Wang et al.)

In general, Krylov subspace methods are only feasible in combination with a
preconditioner.  In~\cite{WanSP07} only factorized preconditioners of the form
$A\approx CC^T$ can be used such that instead of $Ax=b$ the preconditioned
system $C^{-1}AC^{-T}y=C^{-1}b$ is solved. In this case the system matrix
remains symmetric. While preconditioners like (incomplete) Cholesky
factorizations have this form, other important classes like (algebraic)
multigrid do not.
A major difference of the method presented here from RMINRES
is that it allows for a greater variety of preconditioners.
%that are applicable in the standard MINRES method.
The only restrictions on the preconditioner $M^{-1}$ are that it has to be
  self-adjoint and positive-definite, and that its inverse has
  to be known; see the discussion at the end of section~\ref{sec:minres:ritz}
  for more details.
  While this excludes the popular class of multigrid preconditioners with a fixed
  number of cycles, full multigrid preconditioners are admissible.
  To the best knowledge of the authors, no such method has been considered
  before.
  \rev{Note that the requirement of a self-adjoint and positive-definite
    preconditioner $M^{-1}$ is common in the context of methods for self-adjoint
    problems (e.g., CG and MINRES) because it allows to change the inner product
    implicitly. With such a preconditioner, the inertia of $A$ is preserved in
    $M^{-1}A$. Deflation is able to remedy the problem to a certain extent,
    e.g., if $A$ only has a few negative eigenvalues.
  }

\rev{Moreover, general inner products are considered which facilitate the
incorporation of arbitrary preconditioners and allow to exploit the
self-adjointness of a problem algorithmically when its natural inner product is
used.}
  This leads to an efficient
  three-term recurrence with the MINRES method instead of a costly full
  orthogonalization in GMRES\@.  One important example of problems that are
  self-adjoint with respect to a non-Euclidean inner product are nonlinear
  Schr\"odinger equations, presented in more detail in section~\ref{sec:nls}.
  General inner products have been considered before; see, e.g.,
    Freund, Golub, and Nachtigal~\cite{FreGN92} or Eiermann, Ernst, and
  Schneider~\cite{EieES00}. Naturally, problems which are Hermitian (with
respect to the Euclidean inner product) also benefit from the results in this
work.

In many of the previous approaches, the underlying Krylov subspace method
itself has to be modified for including deflation; see, e.g., the
  \emph{modified MINRES method} of Wang, de~Sturler, and
  Paulino~\cite[algorithm~1]{WanSP07}. In contrast, the work in the present
  paper separates the deflation methodology from the Krylov subspace method.
Deflation can thus be implemented algorithmically as a wrapper around any
existing MINRES code, e.g., optimized high-performance variants.  The notes on
the implementation in sections~\ref{sec:minres:defl} and~\ref{sec:minres:ritz}
discuss efficient realizations thereof.

For the sake of clarity, restarting -- often used to mitigate memory
  constraints -- is not explicitly discussed in the present paper. However, as
  noted in section~\ref{sec:minres:ritz}, it can be added easily to the algorithm without
affecting the presented advantages of the method. \rev{Note that the method
in~\cite{WanSP07} does not require restarting because it computes Ritz vectors
from a fixed number of Lanczos vectors (\emph{cycling}), cf.\
section~\ref{sec:minres:ritz}. Since the non-restarted method maintains global
optimality over the entire Krylov subspace (in exact arithmetic), it may exhibit
a more favorable convergence behavior than restarted methods.}

In addition to the deflation of computed Ritz vectors, other vectors can be
included that carry explicit information about the problem in question.  For
example, approximations to eigenvectors corresponding to critical eigenvalues
are readily available from analytic considerations. Applications for this are
plentiful, e.g., flow in porous media considered by Tang et al.~\cite{TanNVE09}
and nonlinear Schr\"odinger equations, see section~\ref{sec:nls}.

%Additionally, the approach can also
%incorporate information that is explicitly provided.
%This is of particular interest for problems which are known to have
%a small number of eigenvalues that harm the Krylov convergence; e.g.,
%eigenvalues of large or small magnitude.

The deflation framework with the properties presented in this paper are
applied in the numerical solution of nonlinear Schr\"odinger equations.
Nonlinear Schr\"odinger equations and their variations are used to describe a
wide variety of physical systems, most notably in superconductivity, quantum
condensates, nonlinear acoustics~\cite{som1979coupled}, nonlinear
optics~\cite{gedalin1997optical}, and hydrodynamics~\cite{nore1993numerical}.
For the solution of nonlinear Schr\"odinger equations with Newton's method, a
linear system has to be solved with the Jacobian operator for each Newton
update. The Jacobian operator is self-adjoint with respect to a non-Euclidean
inner product and indefinite.  In order to be applicable in practice, the
MINRES method can be combined with an AMG-type preconditioner that is able to
limit the number of MINRES iterations to a feasible extent~\cite{SV:2012:OLS}.
Due to the special structure of the nonlinear Schr\"odinger equation, the
Jacobian operator exhibits one eigenvalue that moves to zero when the Newton
iterate converges to a nontrivial solution and is exactly zero in a solution.
Because this situation only occurs in the last step, no linear system has to be
solved with an exactly singular Jacobian operator but the severely
ill-conditioned Jacobian operators in the final Newton steps lead to
convergence slowdown or stagnation in the MINRES method even when a
preconditioner is applied. For the numerical experiments we consider the
Ginzburg--Landau equation, an important instance of nonlinear Schr\"odinger
equations that models phenomena of certain superconductors. We use the proposed
recycling MINRES method and show how it can help to improve the convergence of
the MINRES method. All enhancements of the deflated MINRES method, i.e.,
arbitrary inner products and preconditioners are required for this application.
As a result, the overall time consumption of Newton's method for the
Ginzburg--Landau equation is reduced by roughly 40\%.

The deflated Krylov subspace methods described in this paper are implemented
  in the Python package \emph{KryPy}~\cite{krypy}; solvers for nonlinear
  Schr\"odinger problems are available from \emph{PyNosh}~\cite{pynosh}. Both
  packages are free and open-source software. All results from this paper can
  be reproduced with the help of these packages.

The paper is organized as follows: section~\ref{sec:minres} gives a brief
overview on the preconditioned MINRES method for an arbitrary nonsingular
linear operator that is self-adjoint with respect to an arbitrary inner
product.  The deflated MINRES method is described in
subsection~\ref{sec:minres:defl} while subsection~\ref{sec:minres:ritz}
presents the computation of Ritz vectors and explains their use in the overall
algorithm for the solution of a sequence of self-adjoint linear systems. In
section~\ref{sec:nls} this algorithm is applied to the Ginzburg--Landau
equation. Subsections~\ref{sec:nls:principal} and~\ref{sec:nls:gl} deal with
the numerical treatment of nonlinear Schr\"odinger equations in general and the
Ginzburg--Landau equation in particular. In subsection~\ref{sec:nls:exp}
numerical results for typical two- and three-dimensional setups are presented.

%Nonetheless,
%the problem of finding solutions to nonlinear problems
%is thus relaxed to finding good initial guesses for
%solutions, which is in many cases almost evenly hard to achieve.
%A common approach for this is to identify system parameters $p$
%such that a solution $\psi_0$ can be inferred analytically in
%one particular parameter setting.
%Such parameters could specify, e.g.,
%the sample geometry, external fields, boundary conditions,
%or could be artificially introduced to the problem.

%\begin{example}
%Given $\Omega=\R$,
%consider the steady state for the Nonlinear Schr\"odinger
%equation
%\[
%0 = -\Delta \psi + \kappa |\psi|^2 \psi
%\]
%with unknown $\psi:\Omega\to\C$ and a parameter $\kappa\in\R$.
%\end{example}

%Starting from the solution tuple $(\psi_0,p_0)$,
%numerical parameter continuation can be performed
%to generate a sequence $(\psi_k,p_k)_{k=0}^m$ of solutions
%for different parameter values~\cite{beyn2002numerical}
%This requires the solution of a sequence of nonlinear equation systems
%which typically require the solution of a number of linear equation
%systems in each step.
%In order to develop efficient computational tools
%that can unravel
%the dynamics of these complex systems, it is hence essential
%to have an efficient solver for the linearization of the problem.
%This is true in particular for systems with a large number
%of unknowns as appearing, e.g., from the discretization of
%systems on three-dimensional physical domains.
%It is t...

\section{MINRES}
\label{sec:minres}

\subsection{Preconditioned MINRES with arbitrary inner product}

This section presents well-known properties of the preconditioned MINRES
method.  As opposed to ordinary textbook presentations this section
incorporates a general Hilbert space.  For $\KK\in\{\R,\C\}$ let $\vsH$ be a
$\KK$-Hilbert space with inner product $\ipdots[\vsH]$ and induced norm
$\nrm[\vsH]{\cdot}$. Throughout this paper the inner product $\ipdots[\vsH]$ is
linear in the first and anti-linear in the second argument and we define
$\vsL(\vsH)\DEF\{\oiL:\vsH\lra\vsH~|~\oiL~\text{is linear and bounded}\}$.
The vector space of $k$-by-$l$ matrices is denoted by $\KK^{k,l}$. We
wish to obtain $\vx\in\vsH$ from
\begin{equation}\label{eq:LS}
    \oiA \vx = \vb
\end{equation}
where $\oiA \in\vsL(\vsH)$ is $\ipdots[\vsH]$-self-adjoint and
invertible and $\vb\in\vsH$. The self-adjointness implies that
the spectrum $\sigma(\oiA)$ is real. However, we do not assume that $\oiA$ is
definite.

If an initial guess $\vx_0\in\vsH$ is given, we can approximate $\vx$ by
iterates
\begin{equation}\label{eq:iterates}
    \vx_n = \vx_0 + \vy_n \quad \text{with} \quad \vy_n\in \K_n(\oiA,\vr_0)
\end{equation}
where $\vr_0= \vb-\oiA\vx_0$ is the initial residual and
$\K_n(\oiA,\vr_0)=\spn\{\vr_0,\oiA\vr_0,\ldots,\oiA^{n-1}\vr_0\}$ is the $n$th
Krylov subspace generated with $\oiA$ and $\vr_0$. We concentrate on minimal
residual methods here, i.e., methods that construct iterates of the
form~\eqref{eq:iterates} such that the residual $\vr_n\DEF \vb-\oiA\vx_n$ has
minimal
$\nrm[\vsH]{\cdot}$-norm, that is
\begin{align}\label{eq:minires}
    \nrm[\vsH]{\vr_n}
        &= \nrm[\vsH]{\vb-\oiA\vx_n}
        = \nrm[\vsH]{\vb - \oiA(\vx_0+\vy_n)}
        = \nrm[\vsH]{\vr_0 - \oiA\vy_n} \notag\\
        &= \min_{\vy\in\K_n(\oiA,\vr_0)}{\nrm[\vsH]{\vr_0 - \oiA\vy}}
        = \min_{p\in\Pi_n^0} \nrm[\vsH]{p(\oiA)\vr_0}
\end{align}
where $\Pi_n^0$ is the set of polynomials of degree at most $n$ with $p(0)=1$.
For a general invertible linear operator $\oiA$, the minimization problem in
\eqref{eq:minires} can be solved by the GMRES method~\cite{SaaS86} which is
mathematically equivalent to the MINRES method~\cite{PaiS75} if $\oiA$ is
self-adjoint~\cite[section~2.5.5]{LieS13}.

%The minimization of the residual is equivalent to the condition
%\begin{equation*}\label{eq:orthres}
%    \vr_n \perp_\vsH \oiA \K_n(\oiA,\vr_0).
%\end{equation*}

To facilitate subsequent definitions and statements for general Hilbert spaces,
we use a block notation for inner products that generalizes the common block
notation for matrices:
\begin{dftn}
    For $k,l\in\N$ and two tuples of vectors
    $\vtX=[\vx_1,\ldots,\vx_k]\in\vsH^k$ and $\vtY=[\vy_1,\ldots,\vy_l]\in\vsH^l$
    the product $\ipdots[\vsH]:\vsH^k\times\vsH^l \lra \KK^{k,l}$ is
    defined by
    \[
    \ip[\vsH]{\vtX}{\vtY}\DEF \left[\ip[\vsH]{\vx_i}{\vy_j}\right]_{\substack{i=1,\ldots,k\\j=1,\ldots,l}}.
    \]
\end{dftn}
For the Euclidean inner product and two matrices $\vtX\in\C^{N,k}$ and
$\vtY\in\C^{N,l}$, the product takes the form
$\ip[2]{\vtX}{\vtY}=\vtX^\htp\vtY$.

A block $\vtX\in\vsH^k$ can be right-multiplied with a matrix just as in the
plain matrix case:
\begin{dftn}
    For $\vtX\in\vsH^k$ and
    $\ofZ=\left[z_{ij}\right]_{\substack{i=1,\ldots,k\\j=1,\ldots,l}}\in\KK^{k,l}$,
    right multiplication is defined by
    \[
    \vtX\ofZ \DEF \left[\sum_{i=1}^k z_{ij}\vx_i\right]_{j=1,\ldots,l} \in\vsH^l.
    \]
\end{dftn}

Because the MINRES method and the underlying Lanczos algorithm are often stated
for Hermitian matrices only (i.e., for the Euclidean inner product), we
recall very briefly some properties of the Lanczos algorithm for a linear
operator that is self-adjoint with respect to an arbitrary inner product
$\ipdots[\vsH]$~\cite{EieE01}. If the Lanczos algorithm with inner product $\ipdots[\vsH]$
applied to $\oiA$ and the initial vector $\vv_1=\vr_0/\nrm[\vsH]{\vr_0}$
has completed the $n$th iteration, the Lanczos relation
\begin{align}\label{eq:lanzcos}
    \oiA\vtV_n = \vtV_{n+1} \underline{\ofT}_n
\end{align}
holds, where the elements of $\vtV_{n+1}=[\vv_1,\ldots,\vv_{n+1}]\in\vsH^{n+1}$ form a
$\ipdots[\vsH]$-orthonormal basis of $\K_{n+1}(\oiA,\vr_0)$, i.e.,
$\spn\{\vv_1,\ldots,\vv_{n+1}\}=\K_{n+1}(\oiA,\vr_0)$ and
$\ip[\vsH]{\vtV_{n+1}}{\vtV_{n+1}}=\ofI_{n+1}$. Note that the orthonormality
implies $\nrm[\vsH]{\vtV_{n+1}\vz} = \nrm[2]{\vz}$ for all $\vz\in\KK^{n+1}$. The matrix
$\underline{\ofT}_n\in\R^{n+1,n}$ is real-valued (even if $\KK=\C$), symmetric, and
tridiagonal with
\[
    \underline{\ofT}_k=[\ip[\vsH]{\oiA\vv_i}{\vv_j}]_{\substack{i=1,\ldots,n+1\\j=1,\ldots,n}}.
\]
The $n$th approximation of the solution of the linear system~\eqref{eq:LS}
generated with the MINRES method and the corresponding residual norm,
cf.~\eqref{eq:iterates} and \eqref{eq:minires}, can then be expressed as
\begin{align*}%\label{eq:minires-lanzcos}
    \vx_n &= \vx_0 + \vtV_n \vz_n \quad \text{with} \quad \vz_n\in \KK^n \quad
    \text{and}\\
    \nrm[\vsH]{\vr_n}
        &= \nrm[\vsH]{\vr_0 - \oiA\vtV_n\vz_n}
        = \nrm[\vsH]{\vtV_{n+1}(\nrm[\vsH]{\vr_0} \ve_1 - \underline{\ofT}_n\vz_n)}
        = \nrm[2]{\nrm[\vsH]{\vr_0} \ve_1 - \underline{\ofT}_n\vz_n}.
\end{align*}
By recursively computing a QR decomposition of $\underline{\ofT}_n$, the
minimization problem in \eqref{eq:minires} can be solved without storing the
entire matrix $\underline{\ofT}_n$ and, more importantly, the full Lanczos basis
$\vtV_n$.

Let $N\DEF\dim\vsH<\infty$ and let the elements of $\vtW\in\vsH^N$ form a
$\ipdots[\vsH]$-orthonormal basis of $\vsH$ consisting of eigenvectors of
$\oiA$. Then $\oiA\vtW=\vtW\ofD$
for the diagonal matrix $\ofD=\diag(\lambda_1,\ldots,\lambda_N)$ with $\oiA$'s eigenvalues
$\lambda_1,\ldots,\lambda_N\in\R$ on the diagonal. Let $\vr_0^\vtW\in\KK^N$ be
the representation of $\vr_0$ in the basis $\vtW$, i.e., $\vr_0=\vtW\vr_0^\vtW$.
According to \eqref{eq:minires}, the residual norm of the $n$th approximation
obtained with MINRES can be expressed as
\begin{align}\label{eq:minresbound}
    \nrm[\vsH]{\vr_n}
        &= \min_{{p\in\Pi_n^0}} \nrm[\vsH]{p(\oiA)\vtW\vr_0^\vtW}
        = \min_{{p\in\Pi_n^0}} \nrm[\vsH]{\vtW p(\ofD) \vr_0^\vtW}
        = \min_{{p\in\Pi_n^0}} \nrm[2]{p(\ofD) \vr_0^\vtW} \notag \\
        &\leq \nrm[2]{\vr_0^\vtW} \min_{{p\in\Pi_n^0}} \nrm[2]{p(\ofD)}.
\end{align}
From $\nrm[2]{\vr_0^\vtW}=\nrm[\vsH]{\vtW\vr_0^\vtW}=\nrm[\vsH]{\vr_0}$ and
$\nrm[2]{p(\ofD)} =  \max_{{i\in \{1,\ldots,N\}}} |p(\lambda_i)|$, we
obtain the well-known MINRES worst-case bound for the relative residual
norm~\cite{Gre97,LieT04}
\begin{equation}\label{eq:minreswcbound}
    \frac{\nrm[\vsH]{\vr_n}}{\nrm[\vsH]{\vr_0}}
        \leq \min_{{p\in\Pi_n^0}} \max_{i\in\{1,\ldots,N\}} {|p(\lambda_i)|}.
\end{equation}

This can be estimated even further upon
letting the eigenvalues of $\oiA$ be sorted such that
$\lambda_1\leq\ldots\leq\lambda_s<0<\lambda_{s+1}\leq\ldots\leq\lambda_N$
for a $s\in\N_0$.
%The min-max value in the above bound can then be estimated
By replacing the discrete set of eigenvalues in \eqref{eq:minreswcbound}  by the union of the two intervals
$I^-\DEF[\lambda_1,\lambda_s]$ and $I^+\DEF[\lambda_{s+1},\lambda_N]$,
one gets
\begin{align}\label{eq:minresconv}
    \frac{\nrm[\vsH]{\vr_n}}{\nrm[\vsH]{\vr_0}}
    &\leq \min_{{p\in\Pi_n^0}} \max_{\lambda\in\sigma(\oiA)}|p(\lambda)|
    \leq \min_{{p\in\Pi_n^0}} \max_{\lambda\in I^-\cup I^+}|p(\lambda)| \notag \\
    &\leq 2 \left(
        \frac{\sqrt{|\lambda_1\lambda_N|} - \sqrt{|\lambda_s\lambda_{s+1}|}}
        {\sqrt{|\lambda_1\lambda_N|} + \sqrt{|\lambda_s\lambda_{s+1}|} }
        \right)^{[n/2]},
\end{align}
where $[n/2]$ is the integer part of $n/2$, cf.~\cite{Gre97,LieT04}. Note that
this estimate does not take into account the actual distribution of the
eigenvalues in the intervals $I^-$ and $I^+$. In practice a better
convergence behavior than the one suggested by the estimate above can often be
observed.

In most applications, the MINRES method is only feasible when it is applied with
a preconditioner. Consider the preconditioned system
\begin{equation}\label{eq:precsys}
    \oiM^\inv\oiA\vx = \oiM^\inv\vb
\end{equation}
where $\oiM\in\vsL(\vsH)$ is a $\ipdots[\vsH]$-self-adjoint,
invertible, and positive-definite linear operator. Note that $\oiM^\inv\oiA$ is
not $\ipdots[\vsH]$-self-adjoint but self-adjoint with respect to
the inner product $\ipdots[\oiM]$ defined by $\ip[\oiM]{\vx}{\vy}\DEF
\ip[\vsH]{\oiM\vx}{\vy}= \ip[\vsH]{\vx}{\oiM\vy}$. The MINRES method is then
applied to \eqref{eq:precsys} with the inner product $\ipdots[\oiM]$ and thus
minimizes $\nrm[\oiM]{\oiM^\inv(\vb - \oiA\vx)}$. From an algorithmic
point of view it is worthwhile to note that only the application of $\oiM^\inv$
is needed and the application of $\oiM$ for the inner products can be carried
out implicitly; cf.~\cite[chapter 6]{ElmSW05}. Analogously to
\eqref{eq:minresconv} the convergence bound for the residuals $\til{\vr}_n$
produced by the preconditioned MINRES method is
\[
  \frac{\nrm[\oiM]{\til{\vr}_n}}{\nrm[\oiM]{\oiM^\inv\vr_0}}
  \leq \min_{{p\in\Pi_n^0}} \max_{\mu\in\sigma(\oiM^\inv\oiA)}|p(\mu)|.
\]
Thus the goal of preconditioning is to achieve a more favorable spectrum of
$\oiM^\inv\oiA$ with an appropriate $\oiM^\inv$.

\subsection{Deflated MINRES}
\label{sec:minres:defl}

In many applications even with the aid of a preconditioner the
convergence of MINRES is hampered -- often due to the presence of one or a few
eigenvalues close to zero that are isolated from the remaining spectrum. This
case has recently been studied by Simoncini and Szyld~\cite{SimS13}. Their
analysis and numerical experiments show that an isolated simple eigenvalue can
cause stagnation of the residual norm until a harmonic Ritz value approximates
the outlying eigenvalue well.

Two strategies are well-known in the literature to circumvent the stagnation or
slowdown in the convergence of preconditioned Krylov subspace methods described
above: \emph{augmentation} and \emph{deflation}. In augmented methods the
Krylov subspace is enlarged by a suitable subspace that contains useful
information about the problem. In deflation techniques the operator is modified
with a suitably chosen projection in order to ``eliminate'' components that
hamper convergence; e.g., eigenvalues close to the origin.  For an extensive
overview of these techniques we refer to Eiermann, Ernst, and
Schneider~\cite{EieES00} and the survey article by Simoncini and
Szyld~\cite{SimS07}. Both techniques are closely intertwined and even turn out
to be equivalent in some cases~\cite{GauGLN13}. Here, we concentrate on
deflated methods and first give a brief description of the recycling MINRES
(RMINRES) method introduced by Wang, de Sturler, and Paulino~\cite{WanSP07}
before presenting a slightly different approach.

The RMINRES method by Wang, de Sturler, and Paulino~\cite{WanSP07} is
mathematically equivalent~\cite{GauGLN13} to the application of the MINRES
method to the ``deflated'' equation
\begin{equation}\label{eq:rminressys}
    \oiP_1 \oiA \til{\vx} = \oiP_1 \vb
\end{equation}
where for a given $d$-tuple $\vtU\in\vsH^d$ of linearly independent vectors
(which constitute a basis of the recycling space) and $\vtC\DEF\oiA\vtU$, the
linear operator $\oiP_1\in\vsL(\vsH)$ is defined by $\oiP_1\vx\DEF
\vx - \vtC\ip[\vsH]{\vtC}{\vtC}^\inv\ip[\vsH]{\vtC}{\vx}$.
Note that, although $\oiP_1$ is a $\ipdots[\vsH]$-self-adjoint projection
(and thus an orthogonal projection),
$\oiP_1\oiA$ in general is not. However, as outlined in~\cite[section
4]{WanSP07} an orthonormal basis of the Krylov subspace can still be generated
with MINRES' short recurrences and the operator $\oiP_1\oiA$ because
$\K_n(\oiP_1\oiA,\oiP_1 \vr_0)=\K_n(\oiP_1\oiA\oiP_1^*,\oiP_1\vr_0)$. Solutions of
equation~\eqref{eq:rminressys} are not unique for $d>0$ and thus $\vx$ was
replaced by $\til{\vx}$. To obtain an approximation $\vx_n$ of the original
solution $\vx$ from the approximation $\til{\vx}_n$ generated with MINRES
applied to \eqref{eq:rminressys}, an additional correction has to be applied:
\[
    \vx_n = \til{\oiP}_1\til{\vx}_n + \vtU\ip[\vsH]{\vtC}{\vtC}^\inv\ip[\vsH]{\vtC}{\vb},
\]
where $\til{\oiP}_1\in\vsL(\vsH)$ is defined by
$\til{\oiP}_1\vx\DEF \vx -
\vtU\ip[\vsH]{\vtC}{\vtC}^\inv\ip[\vsH]{\vtC}{\oiA\vx}$.

\medskip
Let us now turn to a slightly different deflation technique for MINRES which we
formulate with preconditioning directly. We will use a projection which has
been developed in the context of the CG method for Hermitian and
positive-definite operators~\cite{Dos88,Nic87,TanNVE09}.
Under a mild assumption, this projection is also well-defined in the
  indefinite case. In contrast to the orthogonal projection $\oiP_1$ used in
  RMINRES, it is not self-adjoint but instead renders the projected operator
  self-adjoint.  This is a natural fit for an integration with the MINRES
method.

Our goal is to use approximations to eigenvectors corresponding to eigenvalues
that hamper convergence in order to modify the operator with a projection.
Consider the preconditioned equation~\eqref{eq:precsys} and assume for a moment
that the elements of $\vtU=[\vu_1,\ldots,\vu_d]\in\vsH^d$ form a
$\ipdots[\oiM]$-orthonormal basis consisting of eigenvectors of $\oiM^\inv\oiA$,
i.e., $\oiM^\inv\oiA\vtU = \vtU\ofD$ with a diagonal matrix
$\ofD=\diag(\lambda_1,\ldots,\lambda_d)\in\R^{d,d}$. Then
$\ip[\oiM]{\vtU}{\oiM^\inv\oiA\vtU}=\ip[\oiM]{\vtU}{\vtU}\ofD=\ofD$ is
nonsingular because we assumed that $\oiA$ is invertible. This motivates the
following definition:

\begin{dftn}\label{def:proj}
    Let $\oiM,\oiA\in\vsL(\vsH)$ be invertible and $\ipdots[\vsH]$-self-adjoint
    operators and let $\oiM$ be positive-definite. Let $\vtU\in\vsH^d$ be such that
    $\ip[\oiM]{\vtU}{\oiM^\inv\oiA\vtU}=\ip[\vsH]{\vtU}{\oiA\vtU}$ is nonsingular. We
    define the projections $\oiP_\oiM,\oiP\in\vsL(\vsH)$
    by
    \begin{equation}\label{eq:projections}
    \begin{split}
        \oiP_\oiM \vx &\DEF \vx - \oiM^\inv\oiA\vtU
            \ip[\oiM]{\vtU}{\oiM^\inv\oiA\vtU}^\inv\ip[\oiM]{\vtU}{\vx}\\
            \text{and}\quad\oiP \vx &\DEF \vx - \oiA\vtU\ip[\vsH]{\vtU}{\oiA\vtU}^\inv\ip[\vsH]{\vtU}{\vx}.%\\
%        \til{\oiP} \vx &\DEF \vx -
%            \vtU\ip[\vsH]{\vtU}{\oiA\vtU}^\inv\ip[\vsH]{\oiA\vtU}{\vx}.
    \end{split}
    \end{equation}
\end{dftn}
The projection $\oiP_\oiM$ is the projection onto
  $\range(\vtU)^{\perp_\oiM}$
along $\range(\oiM^\inv\oiA\vtU)$ whereas $\oiP$ is the projection onto
$\range(\vtU)^{\perp_\vsH}$ along $\range(\oiA\vtU)$.

The assumption in definition~\ref{def:proj} that
$\ip[\oiM]{\vtU}{\oiM^\inv\oiA\vtU}$ is nonsingular holds if and only if
$\range(\oiM^\inv\oiA\vtU)\cap\range(\vtU)^{\perp_\oiM}=\{0\}$ or equivalently
if $\range(\oiA\vtU)\cap\range(\vtU)^{\perp_\vsH}=\{0\}$. As stated above, this
condition is fulfilled if $\vtU$ contains a basis of eigenvectors of
$\oiM^\inv\oiA$ and also holds for good-enough approximations thereof;
see, e.g., the monograph of Stewart and Sun~\cite{SteS90} for a thorough
analysis of perturbations of invariant subspaces.
%Note that this is exactly the case where the RMINRES method may break down,
%cf.~\cite[section 5]{GauGLN13}.
Applying the projection $\oiP_\oiM$ to the preconditioned
equation~\eqref{eq:precsys} yields the deflated equation
\begin{equation}\label{eq:deflsys}
    \oiP_\oiM \oiM^\inv \oiA\til{\vx} = \oiP_\oiM \oiM^\inv \vb.
\end{equation}

The following lemma states some important properties of the operator $\oiP_\oiM
\oiM^\inv \oiA$.
\begin{lem}\label{lem:deflsysprop}
    Let the assumptions in definition~\ref{def:proj} hold. Then
    \begin{enumerate}
        \item \label{lem:deflsysprop:commute}
            $\oiP_\oiM \oiM^\inv = \oiM^\inv \oiP$.
        \item $\oiP\oiA=\oiA\oiP^*$ where $\oiP^*$ is the adjoint operator of
            $\oiP$ with respect to $\ipdots[\vsH]$, defined by
            $\oiP^*\vx=\vx - \vtU\ip[\vsH]{\vtU}{\oiA\vtU}^\inv\ip[\vsH]{\oiA\vtU}{\vx}$.
        \item \label{lem:deflsysprop:selfadj}
            $\oiP_\oiM \oiM^\inv \oiA
            = \oiM^\inv \oiP \oiA
            = \oiM^\inv \oiA \oiP^*$ is self-adjoint with respect to $\ipdots[\oiM]$.
        \item For each initial guess $\til{\vx}_0\in\vsH$, the MINRES method with inner
            product $\ipdots[\oiM]$ applied to equation~\eqref{eq:deflsys} is well
            defined at each iteration until it terminates with a solution of the
            system.
        \item \label{lem:deflsysprop:res} If $\til{\vx}_n$ is the $n$th
            approximation and $\oiP_\oiM\oiM^\inv\vb -
            \oiP_\oiM\oiM^\inv\oiA\til{\vx}_n$ the corresponding residual
            generated by the MINRES method with inner product $\ipdots[\oiM]$
            applied to~\eqref{eq:deflsys} with initial guess $\til{\vx}_0\in\vsH$, then
            the corrected approximation
            \begin{equation}\label{eq:deflsyscorr}
                \vx_n\DEF\oiP^*\til{\vx}_n +
                \vtU\ip[\vsH]{\vtU}{\oiA\vtU}^\inv\ip[\vsH]{\vtU}{\vb}
            \end{equation}
            fulfills
            \begin{equation}\label{eq:deflsysres}
                \oiM^\inv\vb - \oiM^\inv\oiA\vx_n = \oiP_\oiM\oiM^\inv\vb -
            \oiP_\oiM\oiM^\inv\oiA\til{\vx}_n.
            \end{equation}
            (Note that \eqref{eq:deflsysres} also holds for $n=0$.)
    \end{enumerate}
\end{lem}
\begin{proof}
    Statements 1, 2 and the equation in 3 follow from elementary calculations.  Because
    \begin{align*}
        \ip[\oiM]{\oiP_\oiM \oiM^\inv \oiA\vx}{\vy}
        &= \ip[\vsH]{\oiP\oiA\vx}{\vy}
        = \ip[\vsH]{\oiA\vx}{\oiP^*\vy}
        = \ip[\vsH]{\vx}{\oiA\oiP^*\vy}
        = \ip[\vsH]{\vx}{\oiP\oiA\vy}\\
        &= \ip[\oiM]{\vx}{\oiP_\oiM\oiM^\inv\oiA\vy}.
    \end{align*}
    holds for all $\vx,\vy\in\vsH$, the operator $\oiP_\oiM \oiM^\inv \oiA$ is
    self-adjoint with respect to $\ipdots[\oiM]$.

    Statement 4 immediately follows from~\cite[Theorem 5.1]{GauGLN13} and the
    self-adjointness of $\oiP_\oiM \oiM^\inv \oiA$. Note that the referenced
    theorem is stated for the Euclidean inner product but it can easily be
    generalized to arbitrary inner products. Moreover, GMRES is mathematically
    equivalent to MINRES in our case, again due to the self-adjointness.

    Statement 5 follows from 1.\ and 3.\ by direct calculation:
    \begin{align*}
        \oiM^\inv\vb - \oiM^\inv\oiA\vx_n
        &= \oiM^\inv(\vb - \oiA\vtU\ip[\vsH]{\vtU}{\oiA\vtU}^\inv\ip[\vsH]{\vtU}{\vb})
          - \oiM^\inv\oiA\oiP^*\til{\vx}_n\\
        &= \oiM^\inv\oiP\vb - \oiP_\oiM\oiM^\inv\oiA\til{\vx}_n
        = \oiP_\oiM\oiM^\inv\vb -
            \oiP_\oiM\oiM^\inv\oiA\til{\vx}_n.
    \end{align*}
\end{proof}

Now that we know that MINRES is well-defined when applied to the deflated and
preconditioned equation~\eqref{eq:deflsys}, we want to investigate the
convergence behavior in comparison with the original preconditioned
equation~\eqref{eq:precsys}. The following result is well-known for the
positive-definite case; see, e.g., Saad, Yeung, Erhel, and
Guyomarc'h~\cite{SaaYEG00}. The proof is quite canonical and given here for
convenience of the reader.

\begin{lem}\label{lem:deflopspectrum}
    Let the assumptions in definition~\ref{def:proj} and
    $N\DEF\dim\vsH<\infty$ hold.  If
    $\sigma(\oiM^\inv\oiA)=\{\lambda_1,\ldots,\lambda_N\}$ is the spectrum of
    the preconditioned operator $\oiM^\inv\oiA$ and for $d>0$ the
    elements of $\vtU\in\vsH^d$ form a basis of the $\oiM^\inv\oiA$-invariant
    subspace corresponding to the eigenvalues $\lambda_1,\ldots,\lambda_d$ then
    the following holds:
    \begin{enumerate}
        \item The spectrum of the deflated operator $\oiP_\oiM \oiM^\inv \oiA$
            is
            \[
                \sigma(\oiP_\oiM \oiM^\inv \oiA)
                =\{0\}\cup\{\lambda_{d+1},\ldots,\lambda_N\}.
            \]
        \item For $n\geq 0$ let $\vx_n$ be the $n$th corrected approximation
            (cf.\ item~\ref{lem:deflsysprop:res} of lemma~\ref{lem:deflsysprop}) of MINRES
            applied to \eqref{eq:deflsys}  with inner product $\ipdots[\oiM]$
            and initial guess $\til{\vx}_0$. The residuals $\vr_n\DEF \oiM^\inv\vb
            - \oiM^\inv\oiA\vx_n$ then fulfill
            \begin{equation}\label{eq:deflsysconv}
                \frac{\nrm[\oiM]{\vr_n}}{%
                \nrm[\oiM]{\vr_0}}
                \leq \min_{{p\in\Pi_n^0}} \max_{i\in\{d+1,\ldots,N\}} |p(\lambda_i)|.
            \end{equation}
    \end{enumerate}
\end{lem}
\begin{proof}
    \begin{enumerate}
        \item From the definition of $\oiP_\oiM$ in definition~\ref{def:proj} we
            directly obtain $\oiP_\oiM \oiM^\inv\oiA \vtU = 0$ and thus know
            that $0$ is an eigenvalue of $\oiP_\oiM\oiM^\inv\oiA$ with
            multiplicity at least $d$.
            Let the elements of $\vtV\in\vsH^{N-d}$ be orthonormal and such that
            $\oiM^\inv\oiA\vtV=\vtV\ofD_2$ with
            $\ofD_2=\diag(\lambda_{d+1},\ldots,\lambda_N)$. Then
            $\ip[\oiM]{\vtU}{\vtV}=0$ because $\oiM^\inv\oiA$ is self-adjoint
            with respect to $\ipdots[\oiM]$. Thus $\oiP_\oiM\vtV=\vtV$ and the
            statement follows from $\oiP_\oiM\oiM^\inv\oiA\vtV=\vtV\ofD_2$.
        \item Because the residual corresponding to the corrected initial guess
            is $\vr_0=\oiP_\oiM\oiM^\inv(\vb-\oiA\til{\vx}_0) \in
            \range(\vtU)^{\perp_\oiM}=\range(\vtV)$, where $\vtV$ is defined as in
            1., we have $\vr_0=\vtV\vr_0^\vtV$ for a $\vr_0^\vtV\in\KK^{N-d}$.
            Then with $\ofD_2$ as in 1.\ we obtain by using the orthonormality of
            $\vtV$ similar to \eqref{eq:minresbound}:
            \begin{multline*}
                \nrm[\oiM]{\vr_n}
                = \min_{p\in\Pi_n^0}
                    \nrm[\oiM]{p(\oiP_\oiM\oiM^\inv\oiA)\vtV\vr_0^\vtV}
                = \min_{p\in\Pi_n^0}
                    \nrm[\oiM]{\vtV p(\ofD_2) \vr_0^\vtV}\\
                = \min_{p\in\Pi_n^0}
                    \nrm[2]{p(\ofD_2) \vr_0^\vtV}
                \leq \nrm[\oiM]{\vr_0} \min_{p\in\Pi_n^0}
                    \max_{i\in\{d+1,\ldots,N\}} |p(\lambda_i)|.
            \end{multline*}
    \end{enumerate}
\end{proof}

\paragraph{Notes on the implementation}
Item~\ref{lem:deflsysprop:commute} of lemma~\ref{lem:deflsysprop} states that
$\oiP_\oiM \oiM^\inv = \oiM^\inv \oiP$ and thus the MINRES method can be applied
to the linear system
\begin{align} \label{eq:deflsys2}
    \oiM^\inv \oiP \oiA \til{\vx} = \oiM^\inv \oiP \vb
\end{align}
instead of \eqref{eq:deflsys}. When an approximate solution $\til{\vx}_n$ of
\eqref{eq:deflsys2} is satisfactory then the correction~\eqref{eq:deflsyscorr}
has to be applied to obtain an approximate solution of the original
system~\eqref{eq:LS}. Note that neither $\oiM$ nor its inverse $\oiM^\inv$ show
up in the definition of the operator $\oiP$ or its adjoint operator $\oiP^*$
which is used in the correction. Thus the preconditioner $\oiM^\inv$ does not
have to be applied to additional vectors if deflation is used. This can be a
major advantage since the application of the preconditioner operator
$\oiM^\inv$ is the most expensive part in many applications.

The deflation operator $\oiP$ as defined in definition~\ref{def:proj} with
$\vtU\in\vsH^d$ needs to store $2d$ vectors because aside from $\vtU$ also
$\vtC\DEF\oiA\vtU$ should be pre-computed and stored.  Furthermore the matrix
$\ofE\DEF\ip[\vsH]{\vtU}{\vtC}\in\KK^{d,d}$ or its inverse have to be stored.
The adjoint operator $\oiP^*$ needs exactly the same data so no more storage is
required. The construction of $\vtC$ needs $d$ applications of the operator
$\oiA$ but -- as stated above -- no application of the preconditioner operator
$\oiM^\inv$. Because $\ofE$ is Hermitian $d(d+1)/2$ inner products have to be
computed. One application of $\oiP$ or $\oiP^*$ requires $d$ inner products,
the solution of a linear system with the Hermitian $d$-by-$d$ matrix $\ofE$ and
$d$ vector updates. We gather this information in table~\ref{tab:deflcost}.

\begin{table}[bt]
    \centering
    \caption{Storage requirements and computational cost of the projection
    operators $\oiP$ and $\oiP^*$ (cf.~definition~\ref{def:proj} and
    lemma~\ref{lem:deflsysprop}). All vectors are of length $N$,
    i.e., the number of degrees of freedom of the underlying problem.
    Typically, $N\gg d$.}
    \subcaptionbox{Storage requirements.}{%
        \begin{tabular}{lcc}
            \toprule
            & vectors & other\\
            \midrule
            $\vtU$          & $d$ & -- \\
            $\vtC=\oiA\vtU$ & $d$ & -- \\
            $\ofE=\ip[\vsH]{\vtU}{\vtC}$ or $\ofE^\inv$ & -- & $d^2$\\
            \bottomrule
        \end{tabular}
    }\\[3ex]

    \subcaptionbox{Computational cost.}{%
        \begin{tabular}{l>{\centering\arraybackslash}m{1cm}>{\centering\arraybackslash}m{1cm}ccc}
            \toprule
             & \multicolumn{2}{c}{applications of}
                                    & vector     & inner      & solve \\
            % \cmidrule{2-3}
             & $\oiA$ & $\oiM^\inv$ & updates    & products   & with $\ofE$ \\
            \midrule
            Construction of $\vtC$ and $\ofE$
             & $d$    & --          & --         & $d(d+1)/2$ & -- \\
            Application of $\oiP$ or $\oiP^*$
             & --     & --          & $d$        & $d$        & 1 \\
            Application of correction
            & --     & --          & $d$        & $d$       & 1 \\
            \bottomrule
        \end{tabular}
      }
    \label{tab:deflcost}
\end{table}

Instead of correcting the last approximation $\til{\vx}_n$ it is also possible
to start with the corrected initial guess
\begin{equation}\label{eq:x0correct}
    \vx_0=\oiP^* \til{\vx}_0 +
    \vtU\ip[\vsH]{\vtU}{\oiA\vtU}^\inv \ip[\vsH]{\vtU}{\vb}
\end{equation}
and to use $\oiP^*$ as
a right ``preconditioner'' (note that $\oiP^*$ is singular in general). The
difference is mainly of algorithmic nature and will be described very briefly.

For an invertible linear operator $\oiB\in\vsL(\vsH)$ the right-preconditioned
system $\oiA \oiB \vy = \vb$ can be solved for $\vy$ and then the original solution
can be obtained from $\vx=\oiB\vy$. Instead of $\vx_0$ the initial guess $\vy_0\DEF\oiB^\inv\vx_0$ is
used and the initial residual $\vr_0=\vb -\oiA \oiB \vy_0=\vb-\oiA\vx_0$ equals
the residual of the unpreconditioned system. Then iterates
\[
    \vy_n = \vy_0 + \vz_n \quad \text{with} \quad \vz_n\in\K_n(\oiA\oiB,\vr_0)
\]
and $\vx_n\DEF\oiB\vy_n = \vx_0 + \oiB\vz_n$ are constructed such that the residual $\vr_n =
\vb-\oiA\oiB\vy_n = \vb - \oiA\vx_n$ is minimal in $\nrm[\vsH]{\cdot}$. If the operator
$\oiA\oiB$ is self-adjoint the MINRES method can again be used to solve this
minimization problem. Note that $\vy_0$ is not needed and will never be computed
explicitly. The right preconditioning can of course be combined with a positive
definite preconditioner as described in the introduction of
section~\ref{sec:minres}.

We now take a closer look at the case $\oiB=\oiP^*$ which differs from the above
description because $\oiP^*$ is not invertible in general. However, even if the
right-preconditioned system is not consistent (i.e.,
$\vb\notin\range(\oiA\oiP^*)$) the above strategy can be used to solve the
original linear system. With $\vx_0$ from equation~\eqref{eq:x0correct},
let us construct iterates
\begin{equation}\label{eq:iterateright}
    \vx_n = \vx_0 + \oiP^*\vy_n \quad \text{with} \quad \vy_n\in
    \K_n(\oiM^\inv\oiA\oiP^*,\vr_0)
\end{equation}
such that the residual
\begin{equation}\label{eq:resright}
    \vr_n = \oiM^\inv \vb - \oiM^\inv \oiA \vx_n
\end{equation}
has minimal $\nrm[\oiM]{\cdot}$-norm. Inserting \eqref{eq:iterateright} and the
definition of $\vx_0$ into \eqref{eq:resright} yields $\vr_n = \oiM^\inv \oiP
\vb - \oiM^\inv \oiP \oiA \vy_n$ with $\vy_n \in \K_n(\oiM^\inv\oiA\oiP^*,\vr_0)
= \K_n(\oiM^\inv\oiP\oiA,\vr_0)$. The minimization problem is thus the same as
in the case where MINRES is applied to the linear system~\eqref{eq:deflsys2} and
because both the operators and initial vectors coincide the same Lanczos
relation holds. Consequently the MINRES method can be applied for the right
preconditioned system
\begin{equation}\label{eq:deflsysright}
    \oiM^\inv \oiA \oiP^* \vy =\oiM^\inv \vb, \quad \vx =\oiP^*\vy
\end{equation}
with the corrected initial guess $\vx_0$ from equation~\eqref{eq:x0correct}. The
key issue here is that the initial guess is treated as
in~\eqref{eq:iterateright}. A deflated and preconditioned MINRES implementation
following these ideas only needs the operator $\oiP^*$ and the corrected initial
guess $\vx_0$. A correction step at the end then is unnecessary.

\subsection{Ritz vector computation}
\label{sec:minres:ritz}

So far we considered a single linear system and assumed that a basis
for the construction of the projection used in the deflated system is given
(e.g., eigenvectors are given). We now turn to a sequence of preconditioned linear
systems
\begin{equation}\label{eq:seq}
    \oiM_\idx^\inv \oiA_\idx \vx^\idx = \oiM_\idx^\inv \vb^\idx
\end{equation}
where $\oiM_\idx,\oiA_\idx\in\vsL(\vsH)$ are invertible and
self-adjoint with respect to $\ipdots[\vsH]$, $\oiM_\idx$ is positive
definite and $\vx^\idx,\vb^\idx\in\vsH$ for $k\in\{1,\ldots,M\}$. To improve the
readability we use subscript indices for operators and superscript indices for
elements or tuples of the Hilbert space $\vsH$. Such a sequence may arise from a
time dependent problem or a nonlinear equation where solutions are approximated
using Newton's method (cf.~section~\ref{sec:nls}). We now assume that the
operator $\oiM_{(k+1)}^\inv\oiA_{(k+1)}$ only differs slightly from the previous
operator $\oiM_\idx^\inv\oiA_\idx$. Then it may be worthwhile to extract some
eigenvector approximations from the Krylov subspace \emph{and} the deflation
subspace used in the solution of the $k$th system in order to accelerate
convergence of the next system by deflating these extracted approximate
eigenvectors.

For explaining the strategy in more detail we omit the sequence index for a
moment and always refer to the $k$th linear system if not specified otherwise.
Assume that we used a tuple $\vtU\in\vsH^d$ whose elements form a
$\ipdots[\oiM]$-orthonormal basis to set up the projection $\oiP_\oiM$
(cf.~definition~\ref{def:proj}) for the $k$th linear system~\eqref{eq:seq}. We then
assume that the deflated and preconditioned MINRES method, with inner product
$\ipdots[\oiM]$ and initial guess $\til{\vx}_0$, computed a
satisfactory approximate solution after $n$ steps. The MINRES method then
constructs a basis of the Krylov subspace $\K_n(\oiP_\oiM\oiM^\inv\oiA,\vr_0)$ where
the initial residual is $\vr_0=\oiP_\oiM\oiM^\inv(\vb-\oiA\til{\vx}_0)$. Due to
the definition of the projection we know that
$\K_n(\oiP_\oiM\oiM^\inv\oiA,\vr_0)\perp_\oiM \range(\vtU)$ and we now wish to
compute approximate eigenvectors of $\oiM^\inv\oiA$ in the subspace
$\vsS\DEF\K_n(\oiP_\oiM\oiM^\inv\oiA,\vr_0)\oplus \range(\vtU)$. We can
then pick some approximate eigenvectors according to the corresponding
approximate eigenvalues and the approximation quality in order to construct a
projection for the deflation of the $(k+1)$st linear system.

Let us recall the definition of Ritz pairs~\cite{PaiPV95}:

\begin{dftn}\label{def:ritz}
    Let $\vsS\subseteq \vsH$ be a finite dimensional subspace and let
    $\oiB\in\vsL(\vsH)$ be a linear operator.
    $(\vw,\mu)\in\vsS\times\C$ is called a
%    \begin{enumerate}
%        \item
            Ritz pair of $\oiB$ with respect to $\vsS$ and the inner product
            $\ipdots$ if
            \[
                \oiB\vw - \mu\vw \perp_{\ipdots} \vsS.
            \]
%        \item harmonic Ritz pair of $\oiB$ with respect to $\vsS$ and the inner
%            product $\ipdots$ if
%            \[
%                \oiB\vw - \mu \vw \perp_{\ipdots} \oiB\vsS
%            \]
%    \end{enumerate}
\end{dftn}
%\begin{rmk}
%    If $\oiB$ is invertible, then the harmonic Ritz values of $\oiB$ with
%    respect to $\vsS$ are the Ritz values of $\oiB$ with respect to $\oiB\vsS$.
%\end{rmk}

The following lemma gives insight into how the Ritz pairs of the operator
$\oiM^\inv\oiA$ with respect to the Krylov subspace
$\K_n(\oiP_\oiM\oiM^\inv\oiA,\vr_0)$ and the deflation subspace $\range(\vtU)$
can be obtained from data that is available when the MINRES method found a
satisfactory approximate solution of the last linear system.

\begin{lem}\label{lem:ritz}
  \rev{
  Let the following assumptions hold:
  \begin{itemize}
    \item Let $\oiM,\oiA,\vtU,\oiP_\oiM$ be defined as in
      definition~\ref{def:proj} and let $\ip[\oiM]{\vtU}{\vtU}=\ofI_d$.
    \item The Lanczos algorithm with inner product $\ipdots[\oiM]$ applied
      to the operator $\oiP_\oiM\oiM^\inv\oiA$ and an initial vector
      $\vv\in\range(\vtU)^{\perp_\oiM}$ proceeds to the $n$th iteration.
      The Lanczos relation is
      \begin{align}\label{lem:ritz:lanczos}
          \oiP_\oiM\oiM^\inv\oiA \vtV_n = \vtV_{n+1}\underline{\ofT}_n
      \end{align}
      with $\vtV_{n+1}=[\vv_1,\ldots,\vv_{n+1}]\in\vsH^{n+1}$,
      $\ip[\oiM]{\vtV_{n+1}}{\vtV_{n+1}}=\ofI_{n+1}$ and
      $\underline{\ofT}_n=\mat{\ofT_n\\0\dots 0~s_n}\in\R^{n+1,n}$ where
      $s_n\in\R$ is positive and $\ofT_n\in\R^{n,n}$ is tridiagonal, symmetric,
      and real-valued.
    % Note: this excludes the case where range(V_n) is invariant.
  \item
    Let $\vsS\DEF \K_n(\oiP_\oiM\oiM^\inv\oiA,\vv)\oplus\range(\vtU)$
    and $\vw\DEF[\vtV_n,\vtU]\til{\vw}\in\vsS$ for a $\til{\vw}\in\KK^{n+d}$.
  \end{itemize}
  }

    Then $(\vw,\mu)\in\vsS\times\R$ is a
%    \begin{enumerate}
%        \item
            Ritz pair of $\oiM^\inv\oiA$ with respect to $\vsS$ and the inner
            product $\ipdots[\oiM]$ if and only if
            \begin{equation}\label{lem:ritz:ritz}
                \mat{%
                    \ofT_n + \ofB\ofE^\inv\ofB^\htp & \ofB\\
                    \ofB^\htp & \ofE} \til{\vw} = \mu \til{\vw}
            \end{equation}
            where $\ofB\DEF \ip[\vsH]{\vtV_n}{\oiA\vtU}$ and
            $\ofE\DEF\ip[\vsH]{\vtU}{\oiA\vtU}$.
%        \item harmonic Ritz pair of $\oiM^\inv\oiA$ with respect to $\vsS$ and
%            inner product $\ipdots[\oiM]$ if and only if
%            \begin{equation}\label{lem:ritz:harm}
%                \mat{
%                    \ofT_n + \ofB\ofE^\inv\ofB^\htp & \ofB\\
%                    \ofB^\htp & \ofE
%                    } \til{\vw}
%                = \frac{1}{\mu} \ofL^\htp
%                \mat{
%                    \ofI_{n+1} & \underline{\ofB}\\
%                    \underline{\ofB}^\htp & \ofF} \ofL \til{\vw}
%            \end{equation}
%            where
%            $\ofL=\mat{\underline{\ofT}_n & 0\\ \ofE^\inv\ofB^\htp & \ofI_k}$,
%            $\underline{\ofB}=\ip[\vsH]{\vtV_{n+1}}{\oiA\vtU}=
%            \mat{\ofB\\ \ip[\vsH]{\vv_{n+1}}{\oiA\vtU}}$ and
%            $\ofF=\ip[\vsH]{\oiA\vtU}{\oiM^\inv\oiA\vtU}$.
%    \end{enumerate}

    Furthermore, the squared $\nrm[\oiM]{\cdot}$-norm of the
    Ritz residual $\oiM^\inv\oiA\vw - \mu\vw$ is
    \begin{equation}\label{lem:ritz:res}
        \nrm[\oiM]{\oiM^\inv\oiA\vw - \mu\vw}^2
            = (\ofG\til{\vw})^\htp
            \mat{\ofI_{n+1} & \underline{\ofB} & 0 \\
                \underline{\ofB}^\htp & \ofF & \ofE \\
                0 & \ofE & \ofI_d} \ofG\til{\vw}
    \end{equation}
    where
    \begin{align*}
        \underline{\ofB}&=\ip[\vsH]{\vtV_{n+1}}{\oiA\vtU}=
            \mat{\ofB\\ \ip[\vsH]{\vv_{n+1}}{\oiA\vtU}},\\
        \ofF&=\ip[\vsH]{\oiA\vtU}{\oiM^\inv\oiA\vtU}\quad\text{and}\\
        \ofG&=\mat{\underline{\ofT}_n -\mu \underline{\ofI}_n & 0\\
            \ofE^\inv\ofB^\htp & \ofI_d\\
            0 & -\mu\ofI_d
            }\quad\text{with}\quad \underline{\ofI}_n=\mat{\ofI_n\\0}.
    \end{align*}
\end{lem}

\begin{proof}
%    \begin{enumerate}
%        \item
            $(\vw,\mu)$ is a Ritz pair of $\oiM^\inv\oiA$ with respect to
            $\vsS=\range([\vtV_n,\vtU])$ and the inner product $\ipdots[\oiM]$
            if and only if
            \begin{align*}
                &\oiM^\inv\oiA\vw - \mu\vw \perp_\oiM \vsS\\
                \LLRA \qquad & \ip[\oiM]{\vs}{\oiM^\inv\oiA\vw - \mu\vw }
                    = 0 \quad \forall \vs\in\vsS\\
                \LLRA \qquad & \ip[\oiM]{[\vtV_n,\vtU]}{(\oiM^\inv\oiA - \mu\oiI)
                    [\vtV_n,\vtU]}\til{\vw} = 0\\
                \LLRA \qquad & \ip[\oiM]{[\vtV_n,\vtU]}{\oiM^\inv\oiA
                    [\vtV_n,\vtU]}\til{\vw} =
                    \mu \ip[\oiM]{[\vtV_n,\vtU]}{[\vtV_n,\vtU]}\til{\vw}\\
                \LLRA \qquad & \ip[\oiM]{[\vtV_n,\vtU]}{\oiM^\inv\oiA
                    [\vtV_n,\vtU]}\til{\vw} =
                    \mu \til{\vw}
            \end{align*}
            where the last equivalence follows from the orthonormality of $\vtU$ and
            $\vtV_n$ and the fact that $\range(\vtU) \perp_\oiM \K_n(\oiP_\oiM
            \oiM^\inv\oiA,\vv) = \range(\vtV_n)$. We decompose the left hand side as
            \[
                \ip[\oiM]{[\vtV_n,\vtU]}{\oiM^\inv\oiA [\vtV_n,\vtU]}
                = \mat{ \ip[\oiM]{\vtV_n}{\oiM^\inv\oiA\vtV_n}
                    &\ip[\oiM]{\vtV_n}{\oiM^\inv\oiA\vtU}\\
                    \ip[\oiM]{\vtU}{\oiM^\inv\oiA\vtV_n}
                    & \ip[\oiM]{\vtU}{\oiM^\inv\oiA\vtU} }.
            \]
            The Lanczos relation~\eqref{lem:ritz:lanczos} is equivalent to
            \begin{equation}\label{lem:ritz:lanczos2}
                \oiM^\inv\oiA\vtV_n = \vtV_{n+1} \underline{\ofT}_n + \oiM^\inv
                    \oiA\vtU\ip[\vsH]{\vtU}{\oiA\vtU}^\inv
                    \ip[\vsH]{\oiA\vtU}{\vtV_n}
            \end{equation}
            from which we can conclude with the $\ipdots[\oiM]$-orthonormality of
            $[\vtV_{n+1},\vtU]$ that
            \begin{align*}
                \ip[\oiM]{\vtV_n}{\oiM^\inv\oiA\vtV_n}
                &= \ip[\oiM]{\vtV_n}{\vtV_{n+1}}\underline{\ofT}_n +
                    \ip[\oiM]{\vtV_n}{\oiM^\inv\oiA\vtU}
                    \ip[\vsH]{\vtU}{\oiA\vtU}^\inv\ip[\vsH]{\oiA\vtU}{\vtV_n}\\
                &= \ofT_n + \ip[\vsH]{\vtV_n}{\oiA\vtU}
                    \ip[\vsH]{\vtU}{\oiA\vtU}^\inv\ip[\vsH]{\oiA\vtU}{\vtV_n}.
            \end{align*}
            The characterization of Ritz pairs is complete by recognizing that
            $\ofB=\ip[\oiM]{\vtV_n}{\oiM^\inv\oiA\vtU} = \ip[\vsH]{\vtV_n}{\oiA\vtU} =
            \ip[\oiM]{\vtU}{\oiM^\inv\oiA\vtV_n}^\htp$ and
            $\ofE=\ip[\oiM]{\vtU}{\oiM^\inv\oiA\vtU}=\ip[\vsH]{\vtU}{\oiA\vtU}$.

    Only the residual norm equation remains to show. Therefore we compute
    with \eqref{lem:ritz:lanczos2}
    \begin{align*}
        \oiM^\inv\oiA\vw -\mu\vw
        &= \oiM^\inv\oiA [\vtV_n,\vtU] \til{\vw}- \mu [\vtV_n,\vtU] \til{\vw}
        \\
        &= [\vtV_{n+1},\oiM^\inv\oiA\vtU,\vtU]
        \mat{\underline{\ofT}_n- \mu\underline{\ofI}_n & 0\\
            \ofE^\inv\ofB^\htp & \ofI_d\\
            0   & -\mu \ofI_d} \til{\vw} \\
        &= [\vtV_{n+1},\oiM^\inv\oiA\vtU,\vtU] \ofG\til{\vw}.
    \end{align*}
    The squared residual $\nrm[\oiM]{\cdot}$-norm thus is
    \[
        \nrm[\oiM]{\oiM^\inv\oiA\vw -\mu\vw}^2
        = (\ofG\til{\vw})^\htp \ip[\oiM]{[\vtV_{n+1},\oiM^\inv\oiA\vtU,\vtU]}{
            [\vtV_{n+1},\oiM^\inv\oiA\vtU,\vtU]} \ofG\til{\vw}
    \]
    where $\ip[\oiM]{[\vtV_{n+1},\oiM^\inv\oiA\vtU,\vtU]}{
    [\vtV_{n+1},\oiM^\inv\oiA\vtU,\vtU]} = \mat{ \ofI_{n+1} & \underline{\ofB} &
    0 \\ \underline{\ofB}^\htp & \ofF & \ofE \\ 0 & \ofE & \ofI_d}$ can be shown
    with the same techniques as in 1.\ and 2.
\end{proof}
\begin{rmk}
    Lemma~\ref{lem:ritz} also holds for the (rare) case that
    $\K_n(\oiP_\oiM\oiM^\inv\oiA,\vv)$ is an invariant
    subspace of $\oiP_\oiM\oiM^\inv\oiA$ which we excluded for readability reasons. The Lanczos
    relation~\eqref{lem:ritz:lanczos} in this case is $\oiP_\oiM\oiM^\inv\oiA
    \vtV_n = \vtV_n \ofT_n$ which does not change the result.
\end{rmk}

\begin{rmk}
    \label{rmk:harmonic}
    Instead of using Ritz vectors for deflation, alternative approximations to
    eigenvectors are possible. An obvious choice are harmonic Ritz pairs
    $(\vw,\mu)\in\vsS\times\C$ such that
    \begin{equation}
      \label{eq:harmonic}
      \oiB\vw - \mu\vw \perp_{\ipdots} \oiB\vsS,
    \end{equation}
    see~\cite{WanSP07,MorZ98,PaiPV95}. However, in numerical experiments no significant
    difference between regular and harmonic Ritz pairs could be observed, see
    remark~\ref{rmk:harmonic_exp} in section~\ref{sec:nls:exp}.
\end{rmk}

Lemma~\ref{lem:ritz} shows how a Lanczos relation for the operator
$\oiP_\oiM\oiM^\inv\oiA$ (that can be generated implicitly in the
deflated and preconditioned MINRES algorithm, cf.\ end of
section~\ref{sec:minres:defl}) can be used to obtain Ritz pairs of the
``undeflated'' operator $\oiM^\inv\oiA$.  An algorithm for the solution of the
sequence of linear systems~\eqref{eq:seq} as described in the beginning of this
subsection is given in algorithm~\ref{alg:seq}. In addition to the Ritz
vectors, this algorithm can include auxiliary deflation vectors $\vtY^\idx$.

\begin{algorithm}
    \begin{algorithmic}[1]
        \Require For $k\in\{1,\ldots,M\}$ we have:
        \begin{itemize}
            \item $\oiM_\idx\in\vsL(\vsH)$ is $\ipdots[\vsH]$-self-adjoint and
                positive-definite.
                \Comment preconditioner
            \item $\oiA_\idx\in\vsL(\vsH)$ is $\ipdots[\vsH]$-self-adjoint.
                \Comment operator
            \item $\vb^\idx,\vx_0^\idx\in\vsH$.
                \Comment right hand side and initial guess
            \item $\vtY^\idx\in\vsH^{l_\idx}$ for $l_\idx\in\N_0$.
                \Comment auxiliary deflation vectors (may be empty)
        \end{itemize}
        \State $\vtW = [~] \in\vsH^0$
            \Comment no Ritz vectors available in first step
        \For{$k=1 \to M$}
            \State $\vtU =$ orthonormal basis of
                $\spn[\vtW,\vtY^\idx]$ with respect to
                $\ipdots[\oiM_\idx]$. \label{alg:seq:orth}
            \State $\vtC= \oiA_\idx \vtU$,
                $\ofE= \ip[\vsH]{\vtU}{\vtC}$
                \Comment $\oiP^*$ as in lemma~\ref{lem:deflsysprop}
                \label{alg:seq:projsetup}
            \State $\vx_0 = \oiP^*\vx_0^\idx +
                \vtU \ofE^\inv \ip[\vsH]{\vtU}{\vb^\idx}$
                \Comment corrected initial guess
            \State $\vx_n^\idx, \vtV_{n+1}, \underline{\ofT}_n, \ofB= \text{MINRES} (\oiA_\idx, \vb_\idx,
                \oiM_\idx^\inv, \oiP^*, \vx_0, \varepsilon)$

                \begin{tabularx}{0.9\textwidth}{|X}
                    MINRES is applied to
                    $\oiM_\idx^\inv\oiA_\idx\vx^\idx = \oiM_\idx^\inv\vb^\idx$ with
                    inner product $\ipdots[\oiM_\idx]$, right preconditioner
                    $\oiP^*$, initial guess $\vx_0$ and tolerance $\varepsilon>0$,
                    cf.\ section~\ref{sec:minres:defl}. Then:
                    \begin{itemize}
                        \item The approximation $\vx_n^\idx$ fulfills
                            $\nrm[\oiM_\idx]{\oiM_\idx^\inv\vb^\idx -
                            \oiM_\idx^\inv\oiA_\idx\vx_n^\idx} \leq \varepsilon$.
                        \item The Lanzcos relation
                            $\oiM_\idx^\inv\oiA_\idx\oiP^*\vtV_n = \vtV_{n+1}
                            \underline{\ofT}_n$ holds.
                        \item $\ofB = \ip[\vsH]{\vtV_n}{\vtC}$ is generated
                            as a byproduct of the application of $\oiP^*$.
                    \end{itemize}
                \end{tabularx}

            \State $\vw_1,\ldots,\vw_m,\mu_1,\ldots,\mu_m,\rho_1,\ldots,\rho_m =
                \text{Ritz}(\vtU,\vtV_{n+1},\underline{\ofT}_n,\ofB,\vtC,\ofE,\oiM_\idx^\inv)$

                \begin{tabularx}{0.9\textwidth}{|X}
                    Ritz($\ldots$) computes the Ritz pairs
                    $(\vw_j,\mu_j)$ for $j\in\{1,\ldots,m\}$ of
                    $\oiM_\idx^\inv\oiA_\idx$ with respect to $\spn[\vtU,\vtV_n]$
                    and the inner product $\ipdots[\oiM_\idx]$, cf.\ lemma~\ref{lem:ritz}.
                    Then:
                    \begin{itemize}
                        \item $\vw_1,\ldots,\vw_m$ form a
                            $\ipdots[\oiM_\idx]$-orthonormal basis of
                            $\spn[\vtU,\vtV_n]$.
                        \item The residual norms
                            $\rho_j=\nrm[\oiM_\idx]{\oiM_\idx^\inv\oiA_\idx\vw_j-\mu_j\vw_j}$
                            are also returned.
                    \end{itemize}
                \end{tabularx}

            \State $\vtW=[\vw_{i_1},\ldots,\vw_{i_d}]$ for pairwise distinct
                $i_1,\ldots,i_d\in\{1,\ldots,m\}$. \label{alg:seq:pick}

                \begin{tabularx}{0.9\textwidth}{|X}
                    Pick $d$ Ritz vectors according to Ritz value
                    and residual norm.
                \end{tabularx}
                \label{alg:pick}

        \EndFor
    \end{algorithmic}
    \caption{Algorithm for the solution of the sequence of linear
    systems~\eqref{eq:seq}.}
    \label{alg:seq}
\end{algorithm}

\paragraph{Selection of Ritz vectors}

In step~\ref{alg:pick} of algorithm~\ref{alg:seq} up to $m$ Ritz vectors can
be chosen for deflation in the next linear system. It is unclear which choice
leads to optimal convergence. The convergence of MINRES is determined by
the spectrum of the operator and the initial residual in an intricate way.
In most applications one can only use rough convergence bounds of the
type~\eqref{eq:minresconv} which form the basis for certain heuristics.
Popular choices include Ritz vectors corresponding to smallest- or
largest-magnitude Ritz values or smallest Ritz residual norms. No general recipe
can be expected.

\paragraph{Notes on the implementation}

We now comment on the implementational side of the determination and
utilization of Ritz pairs while solving a sequence of linear systems (cf.\
algorithm~\ref{alg:seq}). The solution of a single linear system with the
deflated and preconditioned MINRES method was discussed in
section~\ref{sec:minres:defl}.  Although the MINRES method is based on short
recurrences due to the underlying Lanczos algorithm -- and thus only needs
storage for a few vectors -- we still have to store the full Lanczos basis
$\vtV_{n+1}$ for the determination of Ritz vectors and the Lanczos matrix
$\underline{\ofT}_n\in\R^{n+1,n}$. The storage requirements of the tridiagonal
Lanczos matrix are negligible while storing all Lanczos vectors may be costly.
As customary for GMRES, this difficulty can be overcome by restarting the
  MINRES method after a fixed number of iterations.
  This could be added trivially to
  algorithm~\ref{alg:seq} as well by iterating lines~\ref{alg:seq:orth} to
  \ref{alg:pick} with a fixed maximum number of MINRES iterations for the same
  linear system and the last iterate as initial guess. In this case, the number
  $n$ is interpreted not as the total number of MINRES iterations but as
  the number of MINRES iterations in a restart phase.
\rev{%
As an alternative to restarting, Wang et al.~\cite{WanSP07} suggest to
compute the Ritz vectors in cycles of fixed length $s$. At the end of each
cycle, new Ritz vectors are computed from the previous Ritz vectors and the $s$
Lanczos vectors from the current cycle. All but the last two Lanczos vectors are
then dropped since they are not required for continuing the MINRES iteration.
Therefore, the method in~\cite{WanSP07} is able to maintain global optimality of
the approximate solution with respect to the entire Krylov subspace (in exact
arithmetic), which may lead to faster convergence compared to restarted methods.
Note that a revised RMINRES implementation with performance optimizations has
been published in~\cite{MelSPS10}.
Both restarting and cycling thus provide a way to limit the memory requirements.
However, the quality of computed Ritz vectors and thus the performance as
recycling vectors typically deteriorates.
}

In the experiments in this manuscript, neither restarting nor cycling is
necessary since the preconditioner limits the
number of iterations sufficiently (cf.\ section~\ref{sec:nls:exp}). Deflation
can then be used to further improve convergence by directly addressing parts of
the preconditioned operator's spectrum.  An annotated version of the algorithm
can be found in algorithm~\ref{alg:seq}.
Note that the inner product matrix $\ofB$
  is computed implicitly row-wise in each iteration of MINRES by the
  application of $\oiP^*$ to the last Lanczos vector $\vv_n$ because this
involves the computation of $\ip{\oiA\vtU}{\vv_n}=\ofB^\htp\ve_n$.

\paragraph{Overall computational cost}
An overview of the computational cost of one iteration of
algorithm~\ref{alg:seq} is given in table~\ref{table:minres-cost}.
The computation of one iteration of algorithm~\ref{alg:seq} with $n$ MINRES
steps and $d$ deflation vectors involves $n+d+1$ applications of the
preconditioner $\oiM^{-1}$ and the operator $\oiA$. These steps are typically
very costly and dominate the overall computation time.  This is true for all
variants of recycling Krylov subspace methods.  With this in mind, we would
like to take a closer look at the cost induced by the other elements of the
algorithm.  If the inner products are assumed Euclidean, their computation
accounts for a total of $2N\times(d^2+nd+3d+2n)$~FLOPs.  If the selection strategy
of Ritz vectors for recycling requires knowledge of the respective Ritz
residuals, an additional $2N\times d^2$~FLOPs must be invested. The vector
updates require $2N\times(3/2d^2 + 2nd + 5/2 d + 7n)$~FLOPs, so in total,
without computation of Ritz residuals, $2N\times(5/2d^2 + 3nd + 11/2d +
9n)$~FLOPs are required for one iteration of algorithm~\ref{alg:seq} in
addition to the operator applications.

\rev{%
Comparing the computational cost of the presented method with restarted or
cycled methods is hardly possible. If the cycle length $s$ in~\cite{WanSP07} equals the overall
number of iterations $n$, that method requires $2N\times(6d^2 +
3nd + 3d + 2)$~FLOPs for updating the recycling space. In practice, the methods
show a different convergence behavior because $s\ll n$ and the involved
projections differ, cf.\ section~\ref{sec:minres:defl}.
}

Note that the orthonormalization in line~\ref{alg:seq:orth} is redundant in
exact arithmetic if only Ritz vectors are used and the preconditioner does not
change.
Further note that the orthogonalization requires the application of the
operator $\oiM$, i.e., the inverse of the preconditioner. This operator is not
known in certain cases; e.g., with the application of only a few cycles of an
(algebraic) multigrid preconditioner. Orthogonalizing the columns of $\vtU$
with an inaccurate approximation of $\oiM$ (e.g., the original operator $\oiB$)
will then make the columns of $\vtU$ formally orthonormal with respect to a
different inner product. This may lead to wrong results in the Ritz value
computation. A workaround in the popular case of (algebraic) multigrid
preconditioners is to use so many cycles that $\oiM\approx\oiB$ is fulfilled
with high accuracy. However, this typically leads to a substantial
  increase in computational cost and, depending on the application, may defeat
the original purpose of speeding up the Krylov convergence by recycling.

Similarly, round-off errors may lead to a loss of
  orthogonality in the Lanczos vectors and thus to inaccuracies in the computed
Ritz pairs. Details on this are given in remark~\ref{rem:loss}.

\begin{table}[bt]
    \centering
    \caption{Computational cost for one iteration of algorithm~\ref{alg:seq}
    (lines~\ref{alg:seq:orth}--\ref{alg:seq:pick}) with $n$ MINRES iterations
    and $d$ deflation vectors. The number of computed Ritz vectors also is $d$.
    Operations that do not depend on the dimension $N\DEF \dim \vsH$ are
    neglected.}
    \begin{tabular}{l>{\centering\arraybackslash}m{5ex}>{\centering\arraybackslash}m{6ex}>{\centering\arraybackslash}m{4ex}cc}
        \toprule
        & \multicolumn{3}{c}{Applications of}
                                            & Inner      & Vector     \\
        %\cmidrule{2-4}
        & $\oiA$ & $\oiM^\inv$ & $\oiM$     & products   & updates    \\
        \midrule
        Orthogonalization
        & --     & --          & $d$        & $d(d+1)/2$ & $d(d+1)/2$ \\
        Setup of $\oiP^*$ and $\vx_0$
        & $d$    & --          & --         & $d(d+3)/2$ & $d$       \\
        $n$ MINRES iterations
        & $n+1$    & $n+1$         & --         & $n(d+2)+d$
        & $n(d+7)+d$   \\
        Comp.\ of Ritz vectors
        & --     & $d$         & --         & --      & $d(d+n)$   \\
        (Comp.\ of Ritz res.\ norms)
        & --     & --         & --         & $d^2$      & --\\
        \bottomrule
    \end{tabular}
    \label{table:minres-cost}
\end{table}

\section{Application to nonlinear Schr\"odinger problems}
\label{sec:nls}
Given an open domain $\Omega\subseteq\R^{\{2,3\}}$,
nonlinear Schr\"odinger operators
are typically derived from the minimization
of the Gibbs energy in a corresponding physical system and
have the form
\begin{equation}\label{eq:nls}
%\begin{cases}
\begin{split}
&\Sc: X\to Y,\\
&\Sc(\psi) \dfn (\K + V + g|\psi|^2)\psi \quad \text{in } \Omega%\\[3mm]
%0 = \n \cdot \K\psi \quad \text{on } \partial\Omega,
%\end{cases}
\end{split}
\end{equation}
with $X\subseteq L^2(\Omega)$ being the natural energy space of the problem,
and $Y\subseteq L^2(\Omega)$. If the domain is bounded, the space $X$ may
incorporate boundary conditions appropriate to the physical setting.
The linear operator $\K$ is assumed to be self-adjoint and positive-semidefinite
with respect to $\ipdots[L^2(\Omega)]$,
$V:\Omega\to\R$ is a given scalar potential, and $g>0$ is a given nonlinearity parameter.
A state $\hat{\psi}:\Omega\to\C$ is called a solution of the nonlinear
Schr\"odinger equation if
\begin{equation}\label{eq:schroed}
\Sc(\hat{\psi}) = 0.
\end{equation}
Generally, one is only interested in nontrivial solutions $\hat{\psi}\not\equiv 0$.
The function $\hat{\psi}$ is often referred to as \emph{order parameter} and
its magnitude $|\hat{\psi}|^2$ typically describes a particle density or, more generally, a probability distribution.
Note that, because of
\begin{equation}\label{eq:symm}
\Sc(\exp\{\ii\chi\}\psi) = \exp\{\ii\chi\}\Sc(\psi)
\end{equation}
one solution $\hat{\psi}\in X$ is really just a representative
of the physically equivalent solutions $\{\exp\{\ii\chi\}\hat{\psi}: \chi\in\R\}$.

For the numerical solution of~(\ref{eq:schroed}), Newton's method is
popular for its fast convergence in a neighborhood of a solution:
Given a good-enough initial guess $\psi_0$,
the Newton process generates a sequence of iterates $\psi_k$ which
converges superlinearly towards a solution $\hat{\psi}$ of~(\ref{eq:schroed}).
In each step $k$ of Newton's method, a linear system with the Jacobian
\begin{equation}\label{eq:jacobian}
\begin{split}
&\J(\psi): X \to Y,\\
&\J(\psi)\phi
\dfn \left(\K  + V + 2g|\psi|^2 \right) \phi + g \psi^2 \conj{\phi}.
\end{split}
\end{equation}
of  $\Sc$ at $\psi_k$ needs to be solved.
Despite the fact that states $\psi$ are generally complex-valued,
$\J(\psi)$ is linear only if $X$ and $Y$ are defined as vector spaces over the
field $\R$ with the corresponding inner product
\begin{equation}\label{eq:real-inner-product}
\left\langle\cdot,\cdot\right\rangle_{\R}\dfn \Re\left\langle\cdot,\cdot\right\rangle_{L^2(\Omega)}.
\end{equation}
%\[
%\langle\phi, \psi\rangle_{\R} \dfn \Re\langle\phi, \psi\rangle \quad\forall\phi,\psi\in X.
%\]
This matches the notion that the specific complex argument of the order
parameter is of no physical relevancy since $|r\exp\{\ii\alpha\} \psi|^2 =
|r\psi|^2$ for all $r,\alpha\in\R$, $\psi\in X$ (compare with~(\ref{eq:symm})).

Moreover, the work in~\cite{SAV:2012:NBS} gives a representation of adjoints of
operators of the form~\eqref{eq:jacobian}, from which one can derive
\begin{cor}\label{corollary:j-self-adjoint}
For any given $\psi\in Y$, the Jacobian operator $\J(\psi)$~\eqref{eq:jacobian} is self-adjoint with
respect to the inner product~\eqref{eq:real-inner-product}.
\end{cor}

An important consequence of the independence of states of the complex
argument~\eqref{eq:symm}
is the fact that solutions of equation~(\ref{eq:nls}) form a smooth manifold in $X$. Therefore, the linearization~\eqref{eq:jacobian} in solutions
always has a nontrivial kernel.
Indeed, for any $\psi\in X$
\begin{equation}\label{eq:kernel-nontrivial}
\J(\psi)(\ii\psi)
= \left(\K  + V + 2g|\psi|^2 \right) (\ii\psi) - g \ii \psi^2 \conj{\psi}
= \ii \left(\K  + V + g|\psi|^2 \right) \psi
= \ii \Sc(\psi),
\end{equation}
so for nontrivial solutions $\hat{\psi}\in X$, $\psi\not\equiv 0$, $\Sc(\hat{\psi})=0$,
the dimensionality of the kernel of $\J(\psi)$ is at least 1.

Besides the fact that there is always a zero eigenvalue in a solution
$\hat{\psi}$ and that all eigenvalues are real, not much more can be said about
the spectrum; in general, $\mathcal{J}(\psi)$ is indefinite. The definiteness
depends entirely on the state $\psi$; if $\psi$ is a solution to
\eqref{eq:nls}, it is said to be stable or unstable depending whether or not
$\mathcal{J}(\psi)$ has negative eigenvalues. Typically, solutions with low
Gibbs energies tend to be stable whereas highly energetic solutions tend to be
unstable.  For physical systems in practice, it is uncommon to see more than
ten negative eigenvalues for a given solution state.

\subsection{Principal problems for the numerical solution}
\label{sec:nls:principal}

While the numerical solution of nonlinear systems itself is challenging,
the presence of a singularity in a solution as in \eqref{eq:kernel-nontrivial}
adds two major obstacles for using Newton's method.
\begin{itemize}
  \item Newton's method is guaranteed to converge towards a solution
    $\hat{\psi}$
$Q$-su\-per\-lin\-e\-ar\-ly in the area of attraction only
if $\hat{\psi}$ is nondegenerate, i.e., the Jacobian in
$\hat{\psi}$ is regular.
If the Jacobian operator does have a singularity,
only linear convergence can be guaranteed.
\item While no linear system has to be solved with the exactly singular $\J(\hat{\psi})$,
the Jacobian operator close the solution $\J(\hat{\psi}+\delta \psi)$ will have
at least one eigenvalue of small magnitude,
i.e., the Jacobian system becomes ill-conditioned when approaching a solution.
\end{itemize}

Several approaches have been suggested to deal with this situation,
for a concise survey of the matter, see~\cite{griewank1985solving}.
One of the most used strategies is \emph{bordering} which
suggests extending the original problem $\Sc(\psi)=0$ by a
so-called \emph{phase condition} to pin down the redundancy~\cite{champneys2007numerical},
\begin{equation}\label{eq:bord}
0 = \tilde{\Sc}(\psi,\lambda) \dfn
\begin{pmatrix}
\Sc(\psi)+\lambda y\\
p(x)
\end{pmatrix}.
\end{equation}
If $y$ and $p(\cdot)$ are chosen according to some well-understood
criteria~\cite{keller1983bordering}, the Jacobian systems can be shown to be
well-conditioned throughout the Newton process.  Moreover, the bordering can be
chosen in such a way that the linearization of the extended system is
self-adjoint in the extended scalar product if the linearization of the
original problem is also self-adjoint.  This method has been applied to the
specialization of the Ginzburg--Landau equations~\eqref{eq:GL}
before~\cite{SAV:2012:NBS}, and naturally generalizes to nonlinear
Schr\"odinger equations in the same way.  One major disadvantage of the
bordering approach, however, is that it is not clear how to precondition the
extended system even if a good preconditioner for the original problem is
known.

In the particular case of nonlinear Schr\"odinger equations, the loss of speed
of convergence is less severe than in more general settings.
Note that there would be no slowdown at all if the Newton update $\delta\psi$, given by
\begin{equation}\label{eq:jac-update}
\J(\psi)\delta\psi = - \Sc(\psi),
\end{equation}
was consistently orthogonal to the null space $\ii\hat{\psi}$ close to a solution $\hat{\psi}$.
While this is not generally true, one is at least in the situation that
%the loss of the nonlinear speed of convergence
%close to a solution is hardly observable.
%This is due to the fact that
the Newton update can never be an exact multiple of the direction
of the approximate null space $\ii\psi$. This is because
\[
\J(\psi)(\alpha\ii\psi) = - \Sc(\psi), \quad \alpha\in\R,
\]
together with \eqref{eq:kernel-nontrivial}, is equivalent to
\[
\alpha \ii\Sc(\psi) = - \Sc(\psi)
\]
which can only be fulfilled if $\Sc(\psi)=0$, i.e., if $\psi$ is already a solution.

Consequently, loss of $Q$-superlinear convergence is hardly ever observed in
numerical experiments.  Figure~\ref{fig:newton-res}, for example, shows the
Newton residual for the two- and three-dimensional test setups, both with the
standard formulation and with the bordering~\eqref{eq:bord} as proposed
in~\cite{SAV:2012:NBS}.  Of course, the Newton iterates follow different
trajectories, but the important thing to note is that in both plain and
bordered formulation, the speed of convergence close the solution is
comparable.

The more severe restriction is in the numerical difficulty of solving the
Jacobian systems in each Newton step due to the increasing ill-posedness of the
problem as described above.
However, although the Jacobian has a nontrivial near-null space
close to a solution, the problem is well-defined at all times.
This is because, by self-adjointness, its left near-null space coincides with
the right near-null space, $\spn\{\ii\hat{\psi}\}$, and
the right-hand-side in \eqref{eq:jac-update},
$-\Sc(\psi)$, is orthogonal
to $\ii\psi$ for any $\psi$:
\begin{multline}\label{eq:s0}
\left\langle \ii\psi, S(\psi) \right\rangle_{\R}
= \left\langle \ii\psi, \K(\psi) \right\rangle_{\R}
+ \left\langle \ii\psi, V(\psi) \right\rangle_{\R}
+ \left\langle \ii\psi, g|\psi|^2\psi \right\rangle_{\R}\\
= \Re\left(\ii \langle\psi,\K\psi\rangle_2\right)
+ \Re\left(\ii \langle\psi,V\psi\rangle_2\right)
+ \Re\left(g\ii\left\langle|\psi|^2,|\psi|^2 \right\rangle_2\right)
= 0.
\end{multline}
The numerical problem is hence caused only by the fact that one
eigenvalue approaches the origin as the Newton iterates
approach a solution. The authors propose to handle this difficulty on the
level of the linear solves for the Newton updates
using the deflation framework developed in section~\ref{sec:minres}.
%This will be applied to two representative test setups.

\subsection{The Ginzburg--Landau equation}
\label{sec:nls:gl}

One important instance of nonlinear Schr\"o\-ding\-er equations~\eqref{eq:nls}
is the Ginzburg--Landau equation that
models supercurrent density for extreme-type-II superconductors.
Given an open, bounded domain $\Omega\subseteq\R^{\{2,3\}}$,
the equations are
\begin{equation}\label{eq:GL}
0 =
\begin{cases}
\K\psi - \psi(1-|\psi|^2) \quad\text{in }\Omega,\\[2ex]
\n\cdot(-\ii\bn-\A)\psi\quad\text{on }\partial\Omega.
\end{cases}
\end{equation}
%with the order parameter $\psi\in H^2_{\C}(\Omega)$.
The operator $\K$ is defined as
\begin{equation}\label{eq:kinetic-energy-operator}
\begin{split}
&\K\colon X \to Y,\\
&\K\phi \dfn (-\ii\bn-\A)^2 \phi.
\end{split}
\end{equation}
with the magnetic vector potential $\A\in
H_{\R^d}^2(\Omega)$~\cite{DGP:1993:MAP}.  The operator $\K$ describes the
energy of a charged particle under the influence of a magnetic field
$\B=\bn\times\A$, and can be shown to be Hermitian and positive-semidefinite;
the eigenvalue $0$ is assumed only for $\A\equiv\0$~\cite{SV:2012:OLS}.
Solutions $\hat{\psi}$ of \eqref{eq:GL} describe the density $|\hat{\psi}|^2$
of electric charge carriers and fulfill $0\le|\hat{\psi}|^2\le 1$
pointwise~\cite{DGP:1993:MAP}.  For two-dimensional domains, they typically
exhibit isolated zeros referred to as \emph{vortices}; in three dimensions,
lines of zeros are the typical solution pattern (see
figure~\ref{fig:solutions}).

\paragraph{Discretization}
For the numerical experiments in this paper, a finite-volume-type
discretization is employed~\cite{du2005numerical,SV:2012:OLS}.
Let $\Omega^{(h)}$ be a discretization of $\Omega$ with a triangulation
$\{T_i\}_{i=1}^m$, $\bigcup_{i=1}^m T_i=\Omega^{(h)}$,
and the node-centered Voronoi tessellation
$\{\Omega_k\}_{k=1}^n$, $\bigcup_{k=1}^n \Omega_k=\Omega^{(h)}$.
Let further $\ee_{i,j}$ denote the edge between two nodes $i$, $j$.
The discretized problem is then to find $\psi^{(h)}\in \C^n$ such that
\begin{equation}\label{eq:discr-ginla}
\forall k\in\{1,\dots,n\}:
\quad 0 = \left(S^{(h)}\psi^{(h)}\right)_k \dfn \left(K^{(h)}\psi^{(h)}\right)_k - \psi^{(h)}_k\left(1-|\psi^{(h)}_k|^2\right),
\end{equation}
where the discrete kinetic energy operator $K^{(h)}$ is defined by
\begin{multline}\label{eq:discr-kin-energy}
\forall\phi^{(h)},\psi^{(h)}\in\C^n: \quad \left\langle K^{(h)}\psi^{(h)}, \phi^{(h)}\right\rangle
=\\
\sum_{\text{edges }\ee_{i,j}} \alpha_{i,j}
\left[
\left(\psi^{(h)}_i - U_{i,j}\psi^{(h)}_j\right) \conj{\phi}^{(h)}_i
+
\left(\psi^{(h)}_j - \conj{U_{i,j}}\psi^{(h)}_i\right) \conj{\phi}^{(h)}_j
\right]
\end{multline}
with the discrete inner product
\begin{equation*}
\left\langle\psi^{(h)}, \phi^{(h)}\right\rangle \dfn \sum_{k=1}^n |\Omega_k|\,\psi^{(h)}_k \conj{\phi}^{(h)}_k
\end{equation*}
and edge coefficients $\alpha_{i,j}\in\R$~\cite{SV:2012:OLS}.
The magnetic vector potential $\A$ is incorporated in the so-called \emph{link variables},
\[
U_{i,j} \dfn \exp\left(-\ii\igralnl{\x_j}{\x_i}{\ee_{i,j}\cdot\A(\w)}{\w}\right).
\]
along the edges $\ee_{i,j}$ of the triangulation.

\begin{rmk}\label{rmk:dk}
In matrix form, the operator $K^{(h)}$ is
represented as a product $K^{(h)}=D^{-1}\widehat{K}$ of the diagonal
matrix $D^{-1}$, $D_{i,i}=|\Omega_i|$, and a Hermitian matrix $\widehat{K}$.
\end{rmk}

This discretization preserves a number of invariants of the problem,
e.g., gauge invariance of the type $\tilde{\psi}\dfn\exp\{\ii\chi\}\psi$, $\tilde{\A}\dfn\A+\nabla\chi$
with a given $\chi\in C^1(\Omega)$.
Moreover, the discretized energy operator $K^{(h)}$ is Hermitian and
positive-definite~\cite{SV:2012:OLS}.
Analogous to \eqref{eq:jacobian}, the discretized Jacobian operator at $\psi^{(h)}$ is defined by
\[
  \begin{split}
    &J^{(h)}(\psi^{(h)}): \C^n \to \C^n,\\
    &J^{(h)}(\psi^{(h)}) \phi^{(h)}
    \dfn \left(K^{(h)} - 1 + 2|\psi^{(h)}|^2 \right) \phi^{(h)} + (\psi^{(h)})^2 \conj{\phi^{(h)}}
  \end{split}
\]
where the vector-vector products are interpreted entry-wise.
The discrete Jacobian is self-adjoint with respect to the discrete inner
product
\begin{equation}\label{eq:real-discr-inner}
\left\langle\psi^{(h)}, \phi^{(h)}\right\rangle_\R \dfn \Re\left(\sum_{k=1}^n |\Omega_k|\,\conj{\psi}^{(h)}_k \phi^{(h)}_k\right)
\end{equation}
and the statements~\eqref{eq:kernel-nontrivial}, \eqref{eq:s0} about the null
space carry over from the continuous formulation.

\begin{rmk}[Real-valued formulation]
There is a vector space isomorphism $\alpha:\C^n\to\R^{2n}$ between $\R^{2n}$
and $\C^n$ as vector space over $\R$ given by the basis mapping
\[
\alpha(e_j^{(n)}) = e_j^{(2n)},\quad
\alpha(\ii e_j^{(n)}) = e_{n+j}^{(2n)}.
\]
In particular, note that the dimensionality of $\C^n_\R$ is $2n$.  The
isomorphism $\alpha$ is also isometric with the natural inner product
$\left\langle\cdot,\cdot\right\rangle_{\R}$ of $\C^n_\R$, since for any given
pair $\phi, \psi\in \C^n$ one has
\[
\left\langle
\begin{pmatrix}
\Re\phi\\
\Im\phi\\
\end{pmatrix},
\begin{pmatrix}
\Re\psi\\
\Im\psi\\
\end{pmatrix}
\right\rangle
=
\left\langle
\Re\phi, \Re\psi
\right\rangle
+
\left\langle
\Im\phi, \Im\psi
\right\rangle
=
\left\langle
\phi, \psi
\right\rangle_{\R}.
\]
Moreover, linear operators over $\C^n_\R$ generally have the form
$L\psi = A\psi + B\conj{\psi}$ with some $A, B\in\C^{n\times n}$
and because of
\[
L w = \lambda w \quad\Leftrightarrow\quad (\alpha L \alpha^{-1}) \alpha w = \lambda \alpha w,
\]
the eigenvalues also exactly convey to its
real-valued image $\alpha L \alpha^{-1}$.

This equivalence can be relevant in practice as quite commonly, the original
com\-plex-val\-ued problem in $\C^n$ is implemented in terms $\R^{2n}$. Using
the natural inner product in this space will yield the expected results without
having to take particular care of the inner product.
\end{rmk}

\subsection{Numerical experiments}
\label{sec:nls:exp}

The numerical experiments are performed with the following two setups.
\begin{setup}[2D]\label{ex:2d}
The circle $\Omega_{\textup{2D}}\dfn\{x\in\R^2: \|\x\|_2 < 5\}$ and the
magnetic vector potential $\A(\x)\dfn \bm{m}\times(\x-\x_0) / \|\x-\x_0\|^3$
with $\bm{m}\dfn(0,0,1)^\tp$ and $\x_0\dfn(0,0,5)^\tp$, corresponding to the
magnetic field generated by a dipole at $\x_0$ with orientation $\bm{m}$.  A
Delaunay triangulation for this domain with 3299 nodes was created using
\emph{Triangle}~\cite{shewchuk2002delaunay}.  With the
discrete equivalent of $\psi_0(\x)=\cos(\pi y)$ as initial guess, the Newton
process converges after 27 iterations with a residual of less than $10^{-10}$
in the discretized norm (see figure~\ref{fig:newton-res}).  The final state is
illustrated in figure~\ref{fig:solutions}.
\end{setup}

\begin{setup}[3D]\label{ex:3d}
The three-dimensional L-shape
\[
\Omega_{\text{3D}}\dfn\{\x\in\R^3: \|\x\|_{\infty}<5\} \backslash \R^3_+,
\]
discretized using Gmsh~\cite{geuzaine2009gmsh} with 72166 points.
The chosen magnetic vector field is constant $\B_{\text{3D}}(\x)\dfn 3^{-1/2}(1,1,1)^\tp$,
represented by
the vector potential $\A_{\text{3D}}(\x)\dfn\frac{1}{2} \B_{\text{3D}}\times\x$.
With the discrete equivalent of $\psi_0(\x)=1$,
the Newton process converges after 22 iterations with a residual of less than $10^{-10}$
in the discretized norm (see figure~\ref{fig:newton-res}).
The final state is illustrated in figure~\ref{fig:solutions}.
\end{setup}

All experimental results presented in this section can be reproduced from the
data published with the free and open source Python packages
\emph{KryPy}~\cite{krypy} and \emph{PyNosh}~\cite{pynosh}. \emph{KryPy} contains
an implementation of deflated Krylov subspace methods; e.g.,
algorithm~\ref{alg:seq}. \emph{PyNosh} provides solvers for nonlinear
Schr\"odinger equations including the above test cases.

\begin{figure}
  \centering
  \setlength\figurewidth{0.39\textwidth}
  \setlength\figureheight{0.75\figurewidth}
  % This file was created by matplotlib v0.1.0.
% Copyright (c) 2010--2012, Nico Schl"omer <nico.schloemer@gmail.com>
% All rights reserved.
%
% The lastest updates can be retrieved from
%
% https://github.com/nschloe/matplotlib2tikz
%
% where you can also submit bug reports and leave comments.
%
\begin{tikzpicture}

\begin{semilogyaxis}[
width=\figurewidth,
height=\figureheight,
scale only axis,
xlabel={Newton step},
ylabel={$\|S^{(h)}\|$},
xmin=0, xmax=30,
ymin=1e-12, ymax=1e5
]
\addplot [black]
coordinates {
(0,246.054961360687) (1,47.0170902917901) (2,846.126811188113) (3,249.265576583671) (4,72.7481718084762) (5,20.6971815129195) (6,6.18017769531597) (7,300.367568478137) (8,88.0099886993605) (9,104.56519358744) (10,30.378603812196) (11,8.52030371268952) (12,2.1532891821304) (13,0.479075778905525) (14,0.188444493337976) (15,0.134920077797571) (16,0.0724421525995564) (17,0.122825241906439) (18,0.00834334421669152) (19,0.259502590666373) (20,0.0088039249163069) (21,0.219097960924367) (22,0.00285003768055603) (23,0.00873060278107469) (24,5.79149567245123e-05) (25,9.32292091887457e-06) (26,6.35225665849403e-10) (27,8.8097599286617e-12)
};
\label{pgfplots:nobord}

\addplot [gray,dashed]
coordinates {
(0,246.054961360687) (1,47.0363071589207) (2,905.382187657701) (3,266.825565569082) (4,77.9510110731461) (5,22.2328537741631) (6,6.3405343989359) (7,14083.509536694) (8,4179.84627297465) (9,1236.36048508937) (10,364.855538937414) (11,107.040841229296) (12,30.9262560562714) (13,8.53436276515465) (14,2.08858380237634) (15,0.345313456797847) (16,0.198534466887245) (17,0.0937602991723843) (18,0.0661409046087194) (19,0.0330972782114333) (20,0.00853615234306089) (21,0.0416504334408859) (22,0.00106647100919489) (23,0.0465452364143513) (24,0.000324296348396482) (25,0.00155332050269323) (26,6.06631356955556e-06) (27,7.45160190411373e-07) (28,1.7898747645773e-12)
};
\label{pgfplots:bord}

\end{semilogyaxis}

\end{tikzpicture}
  \hfill
  % This file was created by matplotlib v0.1.0.
% Copyright (c) 2010--2012, Nico Schl"omer <nico.schloemer@gmail.com>
% All rights reserved.
%
% The lastest updates can be retrieved from
%
% https://github.com/nschloe/matplotlib2tikz
%
% where you can also submit bug reports and leave comments.
%
\begin{tikzpicture}

\begin{semilogyaxis}[
width=\figurewidth,
height=\figureheight,
scale only axis,
xlabel={Newton step},
xmin=0, xmax=30,
ymin=1e-12, ymax=1e5,
yticklabels=\empty
]
\addplot [black]
coordinates {
(0,186.691370822136) (1,29.843235045975) (2,17.7215403704804) (3,4.97174632199204) (4,5.28557578902078) (5,55.7307194231444) (6,16.2156924400106) (7,11150.5803737253) (8,3302.4146915914) (9,977.49165000685) (10,288.925946815577) (11,98.9933294794037) (12,28.8148296963305) (13,8.1699013358768) (14,2.18574399355939) (15,0.528310267421839) (16,0.144556911866983) (17,0.0751682284949973) (18,0.0347743608458282) (19,0.00769880532630812) (20,0.000627393495165509) (21,9.87907949317864e-07) (22,5.87229434570361e-12)
};
\addplot [gray,dashed]
coordinates {
(0,186.69137082213584)
(1,29.843235046452712)
(2,17.718319467926399)
(3,4.9713620334188944)
(4,5.2873501451559823)
(5,37.012365358613849)
(6,10.607217762339619)
(7,3.3750005952962443)
(8,4.1465456249240056)
(9,2.0417399913775873)
(10,5.6404677644523922)
(11,1.9598058437983554)
(12,0.49735934046804881)
(13,0.29687561118817563)
(14,0.34536911576887824)
(16,0.13038209326467493)
(17,0.22281827704358415)
(18,0.039995675898197379)
(19,0.090945520748874562)
(20,0.006801607679073637)
(21,0.006644187879165364)
(22,7.6007178308863896e-05)
(23,7.4842616377167088e-07)
(24,1.5562418068232326e-12)
};

\end{semilogyaxis}

\end{tikzpicture}
  \caption{Newton residual history for the two-dimensional setup~\ref{ex:2d} (left)
  and three-dimensional setup~\ref{ex:3d} (right), each with bordering
  %\eqref{pgfplots:bord} % TODO: insert after completion -- hardly compatible with 'externalize'
  and without
  %\eqref{pgfplots:nobord}. % TODO: insert after completion -- hardly compatible with 'externalize'
  With initial guesses
  $\psi_0^{\text{2D}}(\x)=\cos(\pi y)$ and $\psi_0^{\text{3D}}(\x)=1$,
  respectively, the Newton process delivered the solutions as
  highlighted in figure~\ref{fig:solutions} in 22 and 27 steps, respectively.}
  \label{fig:newton-res}
  \setlength\figurewidth{0.3\textwidth}
  \subcaptionbox{Cooper-pair density $|\psi|^2$.}{\begin{tikzpicture}
\begin{axis}
[
axis equal,
enlargelimits=false,
scale only axis,
hide axis=true,
width=\figurewidth,
%colorbar=true,
%point meta min=0,
%point meta max=1,
%colormap/blackwhite,
%colorbar style={ytick={0,0.5,1},yticklabels={$0$,$\frac{1}{2}$,$1$}}
]
\addplot graphics[xmin=-10,xmax=10,ymin=-10,ymax=10]{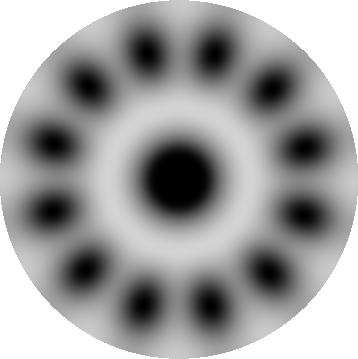};
\end{axis}
\end{tikzpicture}}
  \hspace{0.1\textwidth}
  \subcaptionbox{Cooper-pair density $|\psi|^2$ at the surface of the domain.}{\begin{tikzpicture}
\begin{axis}
[
axis equal,
enlargelimits=false,
scale only axis,
hide axis=true,
width=\figurewidth,
colorbar=true,
point meta min=0,
point meta max=1,
colormap/blackwhite,
colorbar style={ytick={0,0.5,1},yticklabels={$0$,$\frac{1}{2}$,$1$}}
]
\addplot graphics[xmin=-10,xmax=10,ymin=-10,ymax=10]{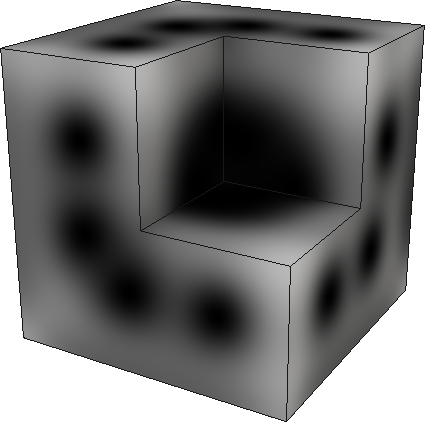};
\end{axis}
\end{tikzpicture}}\\[0.2\figureheight]
  \subcaptionbox{$\arg\psi$.}{\begin{tikzpicture}
\begin{axis}
[
%at={(step22abs.below south east)},
%yshift=-1cm,
%anchor=north east,
axis equal,
enlargelimits=false,
scale only axis,
hide axis=true,
width=\figurewidth,
%colorbar=true,
%point meta min=-3.141592653589793,
%point meta max=3.141592653589793,
%colormap={mymap}{[1pt] rgb(0pt)=(1,0,0); rgb(10pt)=(1,0.9375,0); rgb(11pt)=(0.96875,1,0); rgb(21pt)=(0.03125,1,0); rgb(22pt)=(0,1,0.0625); rgb(32pt)=(0,1,1); rgb(42pt)=(0,0.0625,1); rgb(43pt)=(0.03125,0,1); rgb(53pt)=(0.96875,0,1); rgb(54pt)=(1,0,0.9375); rgb(63pt)=(1,0,0.09375)},
%colorbar style={ytick={-3.141592653589793,0.0,3.141592653589793},yticklabels={$-\pi$,$0$,$\pi$}}
]
\addplot graphics[xmin=-10,xmax=10,ymin=-10,ymax=10]{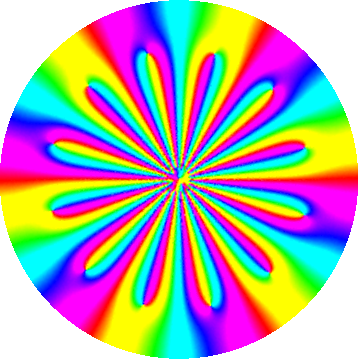};
\end{axis}

\end{tikzpicture}}
  \hspace{0.1\textwidth}
  \subcaptionbox{Isosurface with $|\psi|^2=0.1$ (see (b)), $\arg\psi$ at the back sides of the cube.}{\begin{tikzpicture}
\begin{axis}
[
%at={(step22abs.below south east)},
%yshift=-1cm,
%anchor=north east,
axis equal,
enlargelimits=false,
scale only axis,
hide axis=true,
width=\figurewidth,
colorbar=true,
point meta min=-3.141592653589793,
point meta max=3.141592653589793,
colormap={mymap}{[1pt] rgb(0pt)=(1,0,0); rgb(10pt)=(1,0.9375,0); rgb(11pt)=(0.96875,1,0); rgb(21pt)=(0.03125,1,0); rgb(22pt)=(0,1,0.0625); rgb(32pt)=(0,1,1); rgb(42pt)=(0,0.0625,1); rgb(43pt)=(0.03125,0,1); rgb(53pt)=(0.96875,0,1); rgb(54pt)=(1,0,0.9375); rgb(63pt)=(1,0,0.09375)},
colorbar style={ytick={-3.141592653589793,0.0,3.141592653589793},yticklabels={$-\pi$,$0$,$\pi$}}
]
\addplot graphics[xmin=-10,xmax=10,ymin=-10,ymax=10]{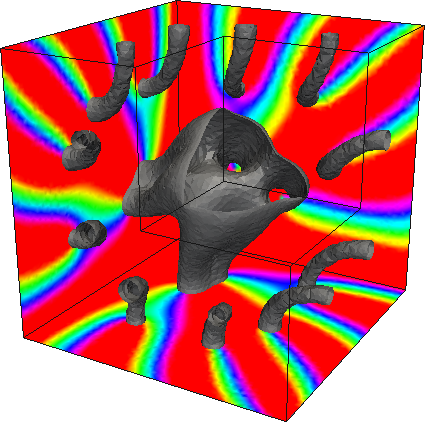};
\end{axis}%
\end{tikzpicture}}
  \caption{Solutions of the test problems as found in the Newton process illustrated in figure~\ref{fig:newton-res}.}
  \label{fig:solutions}
\end{figure}

For both setups, Newton's method was used and the linear
systems~\eqref{eq:jac-update} were solved using MINRES to exploit
self-adjointness of $J^{(h)}$.  Note that it is critical here to use the
natural inner product of the system~\eqref{eq:real-discr-inner}).  All of the
numerical experiments incorporate the preconditioner proposed
in~\cite{SV:2012:OLS} that is shown to bound the number of Krylov iterations
needed to reach a certain relative residual by a constant independent of the
number $n$ of unknowns in the system.

\begin{rmk}
Neither of the above test problems have initial guesses which sit in the cone
of attraction of the solution they eventually converge to. As typical for local
nonlinear solvers, the iterations which do not directly correspond with the
final convergence are sensitive to effects introduced by the discretization or
round-off errors. It will hence be difficult to reproduce precisely the
shown solutions
without exact information about the point coordinates in the discretization
mesh. However, the same general convergence patterns were observed for
  numerous meshes and initial states; the presented solutions shall serve
as examples thereof.
\end{rmk}

\begin{figure}
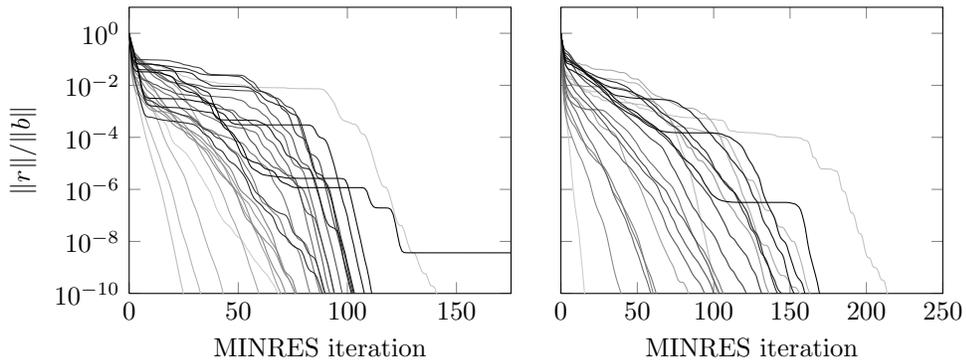
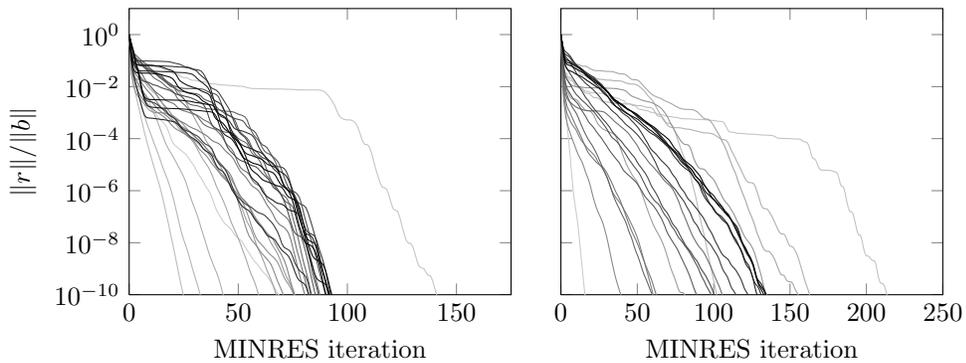
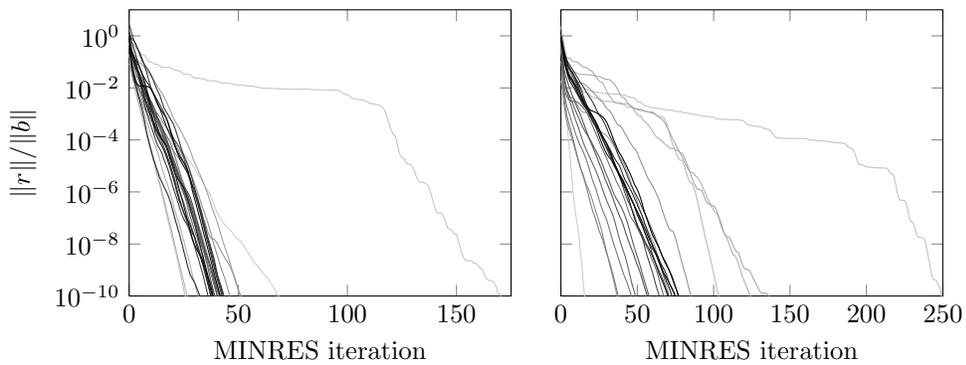

  \centering
  \setlength\figurewidth{0.39\textwidth}
  \setlength\figureheight{0.75\figurewidth}

  \begin{subfigure}[b]{\textwidth}
      \input{figures/2d/newton-vanilla.tex}
      \hfill
      \input{figures/3d/newton-vanilla.tex}
      \caption{Without deflation.}\label{subfig:without}
  \end{subfigure}\\[0.1\figureheight]
  \begin{subfigure}[b]{\textwidth}
      \input{figures/2d/newton-defl00-ix.tex}
      \hfill
      \input{figures/3d/newton-defl00-ix.tex}
      \caption{Deflation of the vector $\ii\psi$.}
      \label{subfig:deflix}
  \end{subfigure}\\[0.1\figureheight]
  \begin{subfigure}[b]{\textwidth}
      \input{figures/2d/newton-defl12.tex}
      \hfill
      \input{figures/3d/newton-defl12.tex}
      \caption{Deflation of 12 Ritz vectors corresponding to the
      Ritz values of smallest magnitude.}
      \label{subfig:defl12}
    \end{subfigure}

    \caption{MINRES convergence histories of all Newton steps for the 2D problem
    (left) and 3D problem (right). The color of the curve corresponds to the
    Newton step: light gray is the first Newton step while black is the last
    Newton step.}
    \label{fig:newton-hist}
\end{figure}

Figure~\ref{fig:newton-hist} shows the relative residuals for all Newton steps
in both the two- and the three-dimensional setup. Note that the residual curves
late in the Newton process (dark gray) exhibit plateaus of stagnation which are
caused by the low-magnitude eigenvalue associated with the near-null space
vector $\ii\hat{\psi}^{(h)}$.

Figure~\ref{subfig:deflix} incorporates the deflation of this vector via
algorithm~\ref{alg:seq} with $\vtY^{(k)} = \ii\psi^{(k,h)}$ where
$\psi^{(k,h)}$ is the discrete Newton approximate in the $k$th step.  The usage
of the preconditioner and the customized inner
product~\eqref{eq:real-discr-inner} is crucial here.  Clearly, the stagnation
effects are remedied and a significantly lower number of iterations is
necessary to reduce the residual norm to $10^{-10}$. While this comes with
extra computational cost per step (cf.\ table~\ref{tab:deflcost}), this cost is
negligible compared to the considerable convergence speedup.

\begin{rmk}
Note that the initial guess $\tilde{x}_0$ is adapted according to
\eqref{eq:x0correct} before the beginning the iteration.  Because of that, the
initial relative residual $\|b-Ax_0\|/\|b-A \tilde{x}_0\|$ cannot generally be
expected to equal $1$ even if $\tilde{x}_0=0$.  In the particular case of
$U=\ii\psi$, however, we have
\[
\vx_0=\oiP^* \til{\vx}_0 +
        \vtU\left\langle\vtU,J(\psi)\vtU\right\rangle_{\R}^\inv \left\langle\vtU,-\Sc(\psi)\right\rangle_{\R}
     = \oiP^* \til{\vx}_0
\]
since $\langle \ii\psi, \Sc(\psi)\rangle=0$~\eqref{eq:s0}, and the initial
relative residual does equal $1$ if $\tilde{x}_0=0$ (cf.\
figure~\ref{subfig:deflix}).  Note that this is not true anymore when more
deflation vectors are added (cf.\ figure~\ref{subfig:defl12}).
\end{rmk}

Towards the end of the Newton process a sequence of very similar linear systems
needs to be solved. We can hence use the deflated MINRES approach described in
algorithm~\ref{alg:seq} where spectral information is extracted from the
previous MINRES iteration and used for deflation in the present process. For
the experiments, those 12 Ritz vectors from the MINRES iteration in Newton step
$k$ which belong to the Ritz values of smallest magnitude were added for
deflation in Newton step $k+1$.  As displayed in figure~\ref{subfig:defl12},
the number of necessary Krylov iterations is further decreased roughly by a
factor of 2.  Note also that in particular, the characteristic plateaus
corresponding to the low-magnitude eigenvalue do no longer occur. This is
particularly interesting since no information about the approximate null space
was explicitly specified, but automatically extracted from previous Newton
steps.

\begin{figure}
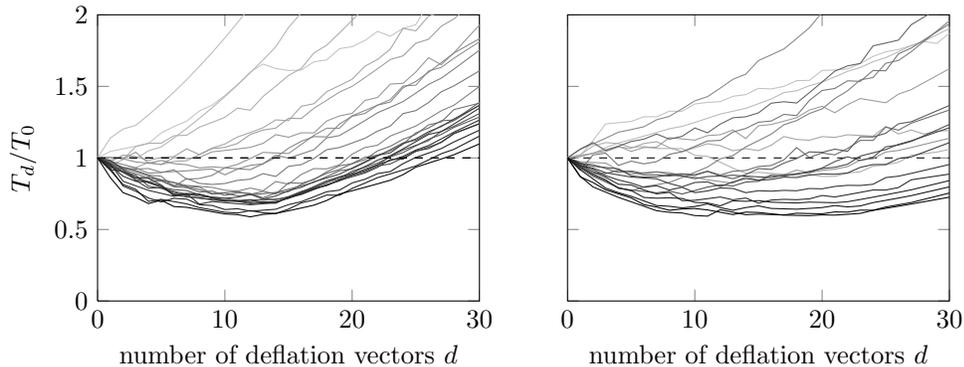

    \centering
    \setlength\figurewidth{0.39\textwidth}
    \setlength\figureheight{0.75\figurewidth}

    \input{figures/2d/newton-efficiency.tex}
    \hfill
    \input{figures/3d/newton-efficiency.tex}

    \caption{Wall-times $T_d$ needed for MINRES solves for the test setups
      (left: 2D; right: 3D) with deflation of those $d$ Ritz vectors from the
      previous Newton step which correspond to the smallest Ritz values.  As in
      the figure~\ref{fig:newton-hist}, light gray lines correspond to steps
      early in the Newton process.  All times are displayed relative to the
      computing time $T_0$ without deflation.  The dashed line at $T_d/T_0=1$
      marks the threshold below which deflation pays off.}
    \label{fig:efficiency}
\end{figure}

As outlined at the end of section~\ref{sec:minres:ritz}, it is a-priori unclear
which choice of Ritz-vectors leads to optimal convergence. Out of the choices
mentioned in section~\ref{sec:minres:ritz}, the smallest-magnitude strategy
performed best in the present application.

Technically, one could go ahead and extract even more Ritz vectors for
deflation in the next step.  However, at some point the extra cost associated
with the extraction of the Ritz vectors (table~\ref{table:minres-cost}) and the
application of the projection operator (table~\ref{tab:deflcost}) will not
justify a further increase of the deflation space.  The efficiency threshold
will be highly dependent on the cost of the preconditioner.  Moreover, it is in
most situations impossible to predict just how the deflation of a particular
set of vectors influences the residual behavior in a Krylov process. For this
reason, one has to resort to numerical experiments to estimate the optimal
dimension of the deflation space.  Figure~\ref{fig:efficiency} shows, again for
all Newton steps in both setups, the wall time of the Krylov iterations as in
figure~\ref{fig:newton-hist} relative to the solution time without deflation.
The experiments show that deflation in the first few Newton steps does not
accelerate the computing speed.  This is due to the fact that the Newton
updates are still significantly large and the subsequent linear systems are too
different from each other in order to take profit from carrying over spectral
information. As the Newton process advances and the updates become smaller, the
subsequent linear systems come closer and deflation of a number of vectors
becomes profitable. Note, however, that there is a point at which the
computational cost of extraction and application of the projection exceeds the
gain in Krylov iterations. For the two-dimensional setup, this value is around
12 while in the three-dimensional case, the minimum roughly stretches from 10
to 20 deflated Ritz vectors.  In both cases, a reduction of effective
computation time by 40\% could be achieved.

\begin{rmk}
    \label{rmk:harmonic_exp}
    Other types of deflation vectors can be considered, e.g., harmonic Ritz
    vectors, see equation~\eqref{eq:harmonic}. In numerical experiments with
    the above test problems we observed that harmonic Ritz vectors resulted in
    a MINRES convergence behavior similar to regular Ritz vectors. \rev{This is
      in accordance with Paige, Parlett, and van der Vorst~\cite{PaiPV95}.}
\end{rmk}

\begin{rmk}
Note that throughout the numerical experiments performed in this paper, the
linear systems were solved up to the relative residual of $10^{-10}$. In
practice, however, one would employ a relaxation scheme as given in,
e.g.,~\cite{eisenstat1996choosing,pernice1998nitsol}.  Those schemes commonly
advocate a relaxed relative tolerance $\eta_k$ in regions of slow convergence,
and a more stringent condition when the speed of convergence accelerates toward
a solution, e.g.,
\[
\eta_k = \gamma \left(\frac{\|F_k\|}{\|F_{k-1}\|}\right)^\alpha
\]
with some $\gamma > 0$, $\alpha>1$.  In the specific case of nonlinear
Schr\"odinger equations, this means that deflation of the near-null vector
$\ii\psi^{(k)}$ (cf.\ figure~\ref{subfig:deflix}) becomes \rev{ineffective} if
$\eta_k$ is \rev{larger} than the stagnation plateau.  The speedup
associated with deflation with a number of Ritz vectors (cf.\
figure~\ref{subfig:defl12}), however, is effective throughout the Krylov
iteration and would hence not be influenced by a premature abortion of the
process.
\end{rmk}

\begin{rmk}\label{rem:loss}
The numerical experiments in this paper were unavoidably affected by round-off
errors. The used MINRES method is based on short recurrences and the
sensitivity to round-off errors may be tremendous. Therefore, a brief
discussion is provided in this remark. A detailed treatment and historical
overview of the effects of finite precision computations on Krylov subspace
methods can be found in the book of Liesen and Strako\v{s}~\cite[sections
5.8--5.10]{LieS13}.
The consequences of round-off errors are manifold and have already been
observed and studied in early works on Krylov subspace methods for linear
algebraic systems, most notably by Lanczos~\cite{Lan52} and Hestenes and
Stiefel~\cite{HesS52}. A breakthrough was the PhD thesis of Paige~\cite{Pai71}
where it was shown that the loss of orthogonality of the Lanczos basis
coincides with the convergence of certain Ritz values.
Convergence may be delayed and the maximal attainable
accuracy, e.g., the smallest attainable residual norm, may be way above machine
precision and above the user-specified tolerance. Both effects heavily depend
on the actual algorithm that is used.
In~\cite{sleijpen2000differences} the impact of certain round-off errors on
the relative residual was analyzed for an unpreconditioned MINRES variant with
the Euclidean inner product. An upper bound on the difference between the
exact arithmetic residual $\vr_n$ and the finite precision residual
$\widehat{\vr}_n$ was given~\cite[formula (26)]{sleijpen2000differences}
\[
  \frac{\nrm[2]{\vr_n - \widehat{\vr}_n}}{\nrm[2]{\vb}}
  \leq \varepsilon \left( 3 \sqrt{3n} \kappa_2(\oiA)^2 + n \sqrt{n}
  \kappa_2(\oiA)\right),
\]
where $\varepsilon$ denotes the \rev{machine epsilon}. The corresponding
bound for GMRES~\cite[formula (17)]{sleijpen2000differences} only involves a
factor of $\kappa_2(\oiA)$ instead of its square. The numerical results
in~\cite{sleijpen2000differences} also indicate that the maximal attainable
accuracy of MINRES is worse than the one of GMRES\@. Thus, if very high accuracy
is required, the GMRES method should be used. An analysis of the stability of
several GMRES algorithms can be found in~\cite{DrkGRS95}.
In order to keep the finite precision Lanczos basis almost orthogonal, a
new Lanczos vector can be reorthogonalized against all previous Lanczos
vectors. The numerical results presented in this paper were computed without
reorthogonalization, i.e., the standard MINRES method. However, all experiments
have also been conducted with reorthogonalization in order to verify that the
observed convergence behavior, e.g., the stagnation phases in
figure~\ref{subfig:without}, are not caused by loss of orthogonality.
\end{rmk}

\section{Conclusions}

For the solution of a sequence of self-adjoint linear systems such as occurring
in Newton process for a large class of nonlinear problems, the authors propose
a MINRES scheme that takes into account spectral information from the previous
linear systems. Central to the approach is the cheap extraction of Ritz vectors
(section~\ref{sec:minres:ritz}) out of a MINRES iteration and the application
of the projection~\eqref{eq:projections}.

As opposed to similar recycling methods previously suggested~\cite{WanSP07},
the projected operator is self-adjoint and is formulated for inner
products other than the $\ell_2$-inner product. This allows for the
incorporation of a wider range of preconditioners than what was previously
possible.
One important restriction that is still remaining is the fact
that for the orthogonalization of the recycling vectors, the inverse of the
preconditioner needs to be known. Unfortunately, this is not the case for some important
classes of preconditioners, e.g., multigrid preconditioners with
a fixed number of cycles. While this
prevents the deflation framework from being universally applicable, the present
work extends the range of treatable problems.

One motivating example for this are nonlinear Schr\"odinger equations (section~\ref{sec:nls}): The occurring linearization is
self-adjoint with respect to a non-Euclidean inner
product~\eqref{eq:real-discr-inner}, and the computation in a three-dimensional
setting is made possible by an AMG preconditioner.  The authors could show that
for the particular case of the Ginzburg--Landau equations, the deflation
strategy reduces the effective run time of a linear solve by up to 40\% (cf.\
figure~\ref{subfig:defl12}).  Moreover, the deflation strategy was shown to
automatically handle the singularity of the problem that otherwise leads to
numerical instabilities.

It is expected that the strategy will perform similarly for other nonlinear
problems.  While adding a number of vectors to the deflation will always lead
to a smaller number of Krylov iterations (and thus less applications of
the operator and the preconditioner), it only comes with extra computational cost in
extracting the Ritz vectors and applying the projection operator;
table~\ref{table:minres-cost} gives a detailed overview of what entities
would need to be balanced.
The optimal
number of deflated Ritz vectors is highly problem-dependent, in particular dependent upon the
computational cost of the preconditioner, and can thus hardly be determined a
priori.

The proposed strategy naturally extends to problems which are not self-adjoint
by choosing, e.g., GMRES as the hosting Krylov method.  For non-self-adjoint
problems, however, the effects of altered spectra on the Krylov convergence is
far more involved than in the self-adjoint case~\cite{murphy1999note}.  This
also makes the choice of Ritz vectors for deflation difficult.  However,
  several heuristics for recycling strategies have been successfully applied to
  non-self-adjoint problems, e.g., by Parks et al.~\cite{ParSMJM06}, Giraud,
Gratton, and Martin~\cite{GirGM07}, Feng, Benner, and Korvink~\cite{FenBK13} as
well as Soodhalter, Szyld, and Xue~\cite{SooSX13}.

\paragraph{Acknowledgments}
The authors wish to thank J\"org~Liesen for his valuable feedback,
Alexander~Schlote for providing experimental results with harmonic Ritz
vectors and the anonymous referees for their helpful remarks.

\bibliographystyle{siam}
\bibliography{andre,nico}

\begin{thebibliography}{10}

\bibitem{champneys2007numerical}
{\sc A.~R. Champneys and B.~Sandstede}, {\em Numerical computation of coherent
  structures}, in Numerical continuation methods for dynamical systems,
  Underst. Complex Syst., Springer, Dordrecht, 2007, pp.~331--358.

\bibitem{ChaS97}
{\sc A.~Chapman and Y.~Saad}, {\em Deflated and augmented {K}rylov subspace
  techniques}, Numer. Linear Algebra Appl., 4 (1997), pp.~43--66.

\bibitem{Stu96}
{\sc E.~de~Sturler}, {\em Nested {K}rylov methods based on {GCR}}, J. Comput.
  Appl. Math., 67 (1996), pp.~15--41.

\bibitem{Dos88}
{\sc Z.~Dost\'{a}l}, {\em Conjugate gradient method with preconditioning by
  projector}, Int. J. Comput. Math., 23 (1988), pp.~315--323.

\bibitem{DrkGRS95}
{\sc J.~Drko{\v{s}}ov{\'a}, A.~Greenbaum, M.~Rozlo{\v{z}}n{\'{\i}}k, and
  Z.~Strako{\v{s}}}, {\em Numerical stability of {GMRES}}, BIT, 35 (1995),
  pp.~309--330.

\bibitem{du2005numerical}
{\sc Q.~Du}, {\em Numerical approximations of the {G}inzburg-{L}andau models
  for superconductivity}, J. Math. Phys., 46 (2005), pp.~095109, 22.

\bibitem{DGP:1993:MAP}
{\sc Q.~Du, M.~D. Gunzburger, and J.~S. Peterson}, {\em Modeling and analysis
  of a periodic {G}inzburg-{L}andau model for type-{${\rm II}$}
  superconductors}, SIAM J. Appl. Math., 53 (1993), pp.~689--717.

\bibitem{EieE01}
{\sc M.~Eiermann and O.~G. Ernst}, {\em Geometric aspects of the theory of
  {K}rylov subspace methods}, Acta Numer., 10 (2001), pp.~251--312.

\bibitem{EieES00}
{\sc M.~Eiermann, O.~G. Ernst, and O.~Schneider}, {\em Analysis of acceleration
  strategies for restarted minimal residual methods}, J. Comput. Appl. Math.,
  123 (2000), pp.~261--292.

\bibitem{eisenstat1996choosing}
{\sc S.~C. Eisenstat and H.~F. Walker}, {\em Choosing the forcing terms in an
  inexact {N}ewton method}, SIAM J. Sci. Comput., 17 (1996), pp.~16--32.
\newblock Special issue on iterative methods in numerical linear algebra
  (Breckenridge, CO, 1994).

\bibitem{ElmSW05}
{\sc H.~C. Elman, D.~J. Silvester, and A.~J. Wathen}, {\em Finite elements and
  fast iterative solvers: with applications in incompressible fluid dynamics},
  Numerical Mathematics and Scientific Computation, Oxford University Press,
  New York, 2005.

\bibitem{FenBK13}
{\sc L.~Feng, P.~Benner, and J.~G. Korvink}, {\em Subspace recycling
  accelerates the parametric macro-modeling of {MEMS}}, Internat. J. Numer.
  Methods Engrg., 94 (2013), pp.~84--110.

\bibitem{FreGN92}
{\sc R.~W. Freund, G.~H. Golub, and N.~M. Nachtigal}, {\em Iterative solution
  of linear systems}, in Acta numerica, 1992, Acta Numer., Cambridge Univ.
  Press, Cambridge, 1992, pp.~57--100.

\bibitem{GauGLN13}
{\sc A.~Gaul, M.~H. Gutknecht, J.~Liesen, and R.~Nabben}, {\em A framework for
  deflated and augmented {K}rylov subspace methods}, SIAM J. Matrix Anal.
  Appl., 34 (2013), pp.~495--518.

\bibitem{krypy}
{\sc A.~Gaul and N.~Schl\"omer}, {\em {KryPy}: {K}rylov subspace methods
  package for {P}ython}.
\newblock \url{https://github.com/andrenarchy/krypy}, Aug. 2013.

\bibitem{pynosh}
\leavevmode\vrule height 2pt depth -1.6pt width 23pt, {\em {PyNosh}: {P}ython
  framework for nonlinear {S}chr\"odinger equations}.
\newblock \url{https://github.com/nschloe/pynosh}, Aug. 2013.

\bibitem{gedalin1997optical}
{\sc M.~Gedalin, T.~Scott, and Y.~Band}, {\em Optical solitary waves in the
  higher order nonlinear {Schr\"odinger} equation}, Phys. Rev. Lett., 78
  (1997), pp.~448--451.

\bibitem{geuzaine2009gmsh}
{\sc C.~Geuzaine and J.-F. Remacle}, {\em Gmsh: {A} 3-{D} finite element mesh
  generator with built-in pre- and post-processing facilities}, Internat. J.
  Numer. Methods Engrg., 79 (2009), pp.~1309--1331.

\bibitem{GirGM07}
{\sc L.~Giraud, S.~Gratton, and E.~Martin}, {\em Incremental spectral
  preconditioners for sequences of linear systems}, Appl. Numer. Math., 57
  (2007), pp.~1164--1180.

\bibitem{Gre97}
{\sc A.~Greenbaum}, {\em Iterative methods for solving linear systems}, vol.~17
  of Frontiers in Applied Mathematics, Society for Industrial and Applied
  Mathematics (SIAM), Philadelphia, PA, 1997.

\bibitem{griewank1985solving}
{\sc A.~Griewank}, {\em On solving nonlinear equations with simple
  singularities or nearly singular solutions}, SIAM Rev., 27 (1985),
  pp.~537--563.

\bibitem{HesS52}
{\sc M.~R. Hestenes and E.~Stiefel}, {\em Methods of conjugate gradients for
  solving linear systems}, J. Research Nat. Bur. Standards, 49 (1952),
  pp.~409--436.

\bibitem{keller1983bordering}
{\sc H.~B. Keller}, {\em The bordering algorithm and path following near
  singular points of higher nullity}, SIAM J. Sci. Statist. Comput., 4 (1983),
  pp.~573--582.

\bibitem{KilS06}
{\sc M.~E. Kilmer and E.~de~Sturler}, {\em Recycling subspace information for
  diffuse optical tomography}, SIAM J. Sci. Comput., 27 (2006), pp.~2140--2166.

\bibitem{Lan52}
{\sc C.~Lanczos}, {\em Solution of systems of linear equations by
  minimized-iterations}, J. Research Nat. Bur. Standards, 49 (1952),
  pp.~33--53.

\bibitem{LieS13}
{\sc J.~Liesen and Z.~Strako{\v{s}}}, {\em Krylov subspace methods. Principles
  and analysis}, Numerical Mathematics and Scientific Computation, Oxford
  University Press, Oxford, 2013.

\bibitem{LieT04}
{\sc J.~Liesen and P.~Tich{\'y}}, {\em Convergence analysis of {K}rylov
  subspace methods}, GAMM Mitt. Ges. Angew. Math. Mech., 27 (2004),
  pp.~153--173 (2005).

\bibitem{MelSPS10}
{\sc L.~A.~M. Mello, E.~de~Sturler, G.~H. Paulino, and E.~C.~N. Silva}, {\em
  Recycling {K}rylov subspaces for efficient large-scale electrical impedance
  tomography}, Comput. Methods Appl. Mech. Engrg., 199 (2010), pp.~3101--3110.

\bibitem{Mor95}
{\sc R.~B. Morgan}, {\em A restarted {GMRES} method augmented with
  eigenvectors}, SIAM J. Matrix Anal. Appl., 16 (1995), pp.~1154--1171.

\bibitem{Mor05}
\leavevmode\vrule height 2pt depth -1.6pt width 23pt, {\em Restarted
  block-{GMRES} with deflation of eigenvalues}, Appl. Numer. Math., 54 (2005),
  pp.~222--236.

\bibitem{MorZ98}
{\sc R.~B. Morgan and M.~Zeng}, {\em Harmonic projection methods for large
  non-symmetric eigenvalue problems}, Numer. Linear Algebra Appl., 5 (1998),
  pp.~33--55.

\bibitem{murphy1999note}
{\sc M.~F. Murphy, G.~H. Golub, and A.~J. Wathen}, {\em A note on
  preconditioning for indefinite linear systems}, SIAM J. Sci. Comput., 21
  (2000), pp.~1969--1972 (electronic).

\bibitem{Nic87}
{\sc R.~A. Nicolaides}, {\em Deflation of conjugate gradients with applications
  to boundary value problems}, SIAM J. Numer. Anal., 24 (1987), pp.~355--365.

\bibitem{nore1993numerical}
{\sc C.~Nore, M.~E. Brachet, and S.~Fauve}, {\em Numerical study of
  hydrodynamics using the nonlinear {S}chr\"odinger equation}, Phys. D, 65
  (1993), pp.~154--162.

\bibitem{Pai71}
{\sc C.~C. Paige}, {\em The computation of eigenvalues and eigenvectors of very
  large sparse matrices}, {PhD} thesis, University of London, institute of
  computer science, 1971.

\bibitem{PaiPV95}
{\sc C.~C. Paige, B.~N. Parlett, and H.~A. van~der Vorst}, {\em Approximate
  solutions and eigenvalue bounds from {K}rylov subspaces}, Numer. Linear
  Algebra Appl., 2 (1995), pp.~115--133.

\bibitem{PaiS75}
{\sc C.~C. Paige and M.~A. Saunders}, {\em Solutions of sparse indefinite
  systems of linear equations}, SIAM J. Numer. Anal., 12 (1975), pp.~617--629.

\bibitem{ParSMJM06}
{\sc M.~L. Parks, E.~de~Sturler, G.~Mackey, D.~D. Johnson, and S.~Maiti}, {\em
  Recycling {K}rylov subspaces for sequences of linear systems}, SIAM J. Sci.
  Comput., 28 (2006), pp.~1651--1674.

\bibitem{pernice1998nitsol}
{\sc M.~Pernice and H.~F. Walker}, {\em N{ITSOL}: a {N}ewton iterative solver
  for nonlinear systems}, SIAM J. Sci. Comput., 19 (1998), pp.~302--318
  (electronic).
\newblock Special issue on iterative methods (Copper Mountain, CO, 1996).

\bibitem{SaaS86}
{\sc Y.~Saad and M.~H. Schultz}, {\em G{MRES}: a generalized minimal residual
  algorithm for solving nonsymmetric linear systems}, SIAM J. Sci. Statist.
  Comput., 7 (1986), pp.~856--869.

\bibitem{SaaYEG00}
{\sc Y.~Saad, M.~Yeung, J.~Erhel, and F.~Guyomarc'h}, {\em A deflated version
  of the conjugate gradient algorithm}, SIAM J. Sci. Comput., 21 (2000),
  pp.~1909--1926.
\newblock Iterative methods for solving systems of algebraic equations (Copper
  Mountain, CO, 1998).

\bibitem{SAV:2012:NBS}
{\sc N.~Schl{\"o}mer, D.~Avitabile, and W.~Vanroose}, {\em Numerical
  bifurcation study of superconducting patterns on a square}, SIAM J. Appl.
  Dyn. Syst., 11 (2012), pp.~447--477.

\bibitem{SV:2012:OLS}
{\sc N.~Schl{\"o}mer and W.~Vanroose}, {\em An optimal linear solver for the
  {J}acobian system of the extreme type-{II} {G}inzburg-{L}andau problem}, J.
  Comput. Phys., 234 (2013), pp.~560--572.

\bibitem{shewchuk2002delaunay}
{\sc J.~R. Shewchuk}, {\em Delaunay refinement algorithms for triangular mesh
  generation}, Comput. Geom., 22 (2002), pp.~21--74.
\newblock 16th ACM Symposium on Computational Geometry (Hong Kong, 2000).

\bibitem{SimS07}
{\sc V.~Simoncini and D.~B. Szyld}, {\em Recent computational developments in
  {K}rylov subspace methods for linear systems}, Numer. Linear Algebra Appl.,
  14 (2007), pp.~1--59.

\bibitem{SimS13}
\leavevmode\vrule height 2pt depth -1.6pt width 23pt, {\em On the superlinear
  convergence of {MINRES}}, in Numerical Mathematics and Advanced Applications
  2011, A.~Cangiani, R.~L. Davidchack, E.~Georgoulis, A.~N. Gorban,
  J.~Levesley, and M.~V. Tretyakov, eds., Springer Berlin Heidelberg, 2013,
  pp.~733--740.

\bibitem{sleijpen2000differences}
{\sc G.~L.~G. Sleijpen, H.~A. van~der Vorst, and J.~Modersitzki}, {\em
  Differences in the effects of rounding errors in {K}rylov solvers for
  symmetric indefinite linear systems}, SIAM J. Matrix Anal. Appl., 22 (2000),
  pp.~726--751 (electronic).

\bibitem{som1979coupled}
{\sc B.~K. Som, M.~R. Gupta, and B.~Dasgupta}, {\em Coupled nonlinear
  {S}chr\"odinger equation for {L}angmuir and dispersive ion acoustic waves},
  Phys. Lett. A, 72 (1979), pp.~111--114.

\bibitem{SooSX13}
{\sc K.~M. Soodhalter, D.~B. Szyld, and F.~Xue}, {\em Krylov subspace recycling
  for sequences of shifted linear systems}, 2013.
\newblock arXiv:1301.2650v3.

\bibitem{SteS90}
{\sc G.~W. Stewart and J.~G. Sun}, {\em Matrix perturbation theory}, Computer
  Science and Scientific Computing, Academic Press Inc., Boston, MA, 1990.

\bibitem{TanNVE09}
{\sc J.~M. Tang, R.~Nabben, C.~Vuik, and Y.~A. Erlangga}, {\em Comparison of
  two-level preconditioners derived from deflation, domain decomposition and
  multigrid methods}, J. Sci. Comput., 39 (2009), pp.~340--370.

\bibitem{WanSP07}
{\sc S.~Wang, E.~de~Sturler, and G.~H. Paulino}, {\em Large-scale topology
  optimization using preconditioned {K}rylov subspace methods with recycling},
  Internat. J. Numer. Methods Engrg., 69 (2007), pp.~2441--2468.

\end{thebibliography}

\end{document}